\documentclass[a4paper,11pt]{amsart}
\usepackage{amsmath}
\usepackage{mathtools}
\usepackage{graphicx}
\usepackage{mathrsfs}
\usepackage{natbib}
\usepackage[breaklinks=true]{hyperref}
\usepackage{breakcites}

\usepackage[left=2.80cm,right=2.80cm]{geometry}  
\newtheorem{theorem}{Theorem}[section]
\newtheorem{lemma}[theorem]{Lemma}
\newtheorem{corollary}[theorem]{Corollary}
\newtheorem{prop}[theorem]{Proposition}

\theoremstyle{definition}

\theoremstyle{remark}
\newtheorem{remark}[theorem]{Remark}
{\theoremstyle{plain}\newtheorem{assumption}{Assumption}}

\numberwithin{equation}{section}

\newcommand{\real}{\mathbb{R}}
\newcommand{\E}{{\rm E}}
\newcommand{\dd}{{\rm d}}
\newcommand{\normal}{{\rm N}}
\newcommand{\eps}{\varepsilon}
\newcommand{\Uscr}{\mathscr{U}}

\newcommand{\cov}{{\rm Cov}}

\newcommand{\abs}[1]{\lvert#1\lvert} 
\newcommand{\norm}[1]{\lVert#1\rVert} 

\begin{document}

\title[]{The standard cure model with a linear hazard}

\author[]{Emil Aas Stoltenberg}
\address{}
\curraddr{Sofies gate 75A, 0454 Oslo, Norway}
\email{emilstoltenberg@gmail.com}
\thanks{I am grateful to Sven Ove Samuelsen for mentioning cure models with linear hazards for me;  to Per~A.~Mykland for teaching me the tools used to solve the problem of this paper, and for constructively criticising my use of them; and to Nils Lid Hjort for reading the manuscript and making me aware of my hazardous shortcuts. I would also like to thank the United States National Science Foundation under grants DMS 17-13118 and DMS-2015530 (Lan Zhang), and DMS 17-13129 and DMS-2015544 (Per A.~Mykland), for the support during my stay in Chicago (2018--2019) where parts of the research for this technical report were carried out.}



\date{\today}

\dedicatory{\sc Department of Mathematics, University of Oslo}

\begin{abstract}
In this paper we introduce a mixture cure model with a linear hazard rate regression model for the event times. Cure models are statistical models for event times that take into account that a fraction of the population might never experience the event of interest, this fraction is said to be {`}cured{'}. The population survival function in a mixture cure model takes the form $S(t) = 1 - \pi + \pi\exp(-\int_0^t\alpha(s)\,\dd s)$, where $\pi$ is the probability of being susceptible to the event under study, and $\alpha(s)$ is the hazard rate of the susceptible fraction. We let both $\pi$ and $\alpha(s)$ depend on possibly different covariate vectors $X$ and $Z$. The probability $\pi$ is taken to be the logistic function $\pi(X^{\prime}\gamma) = 1/\{1+\exp(-X^{\prime}\gamma)\}$, while we model $\alpha(s)$ by Aalen{'}s linear hazard rate regression model. This model postulates that a susceptible individual has hazard rate function $\alpha(t;Z) = \beta_0(t) + \beta_1(t)Z_1 + \cdots + Z_{q-1}\beta_{q-1}(t)$ in terms of her covariate values $Z_1,\ldots,Z_{q-1}$. An EM-algorithm for estimating $\gamma$ and the cumulatives $\int_0^{t}\beta_1(s)\,\dd s,\ldots,\int_0^{t}\beta_{q-1}(s)\,\dd s$ is introduced. The large-sample properties of these estimators are studied by way of parametric models that tend to a semiparametric model as a parameter $K \to \infty$. For each model in the sequence of parametric models, we assume that the data generating mechanism is parametric, thus simplifying the derivation of the estimators, as well as the proofs of consistency and limiting normality. Finally, we use contiguity techniques to switch back to assuming that the data stem from the semiparametric model. This technique for deriving and studying estimators in non- and semiparametric settings has previously been studied and employed in the high-frequency data literature, but seems to be novel in survival analysis.     
\end{abstract}

\maketitle

\section{Introduction}\label{sec::intro} Cure models are statistical models for event times that take into account that a fraction of the population might never experience the event of interest. This fraction of the population is referred to as cured, or nonsusceptible. In the most common cure model construction the hazard rate takes the form   
\begin{equation}
\alpha(t,U) = U \alpha(t) ,
\label{eq::curehazard}
\end{equation}    
where $U$ is a Bernoulli random variable with success probability $0 < \pi < 1$. See~\citet{amico2018cure} for a review of the cure model literature. Under~\eqref{eq::curehazard} a random survival time can be thought of as either stemming from a proper survival distribution with hazard rate $\alpha(t)$, with probability $\pi$; or as being constant and equal to infinity, with probability $1 - \pi$. This model is therefore known as the mixture cure model. The survival function $S(t) = \exp(-\int_0^{t}\alpha(s)\,\dd s)$ only applies to the susceptible fraction of the population. The survival function for the entire population is
\begin{equation}
\E\,\exp(-U\int_0^{t}\alpha(s)\,\dd s) = 1 - \pi + \pi \exp(- \int_0^t \alpha(s)\,\dd s).
\notag
\end{equation}    
Assuming that $\exp(- \int_0^t \alpha(s)\,\dd s)$ is a proper survival function, which we do, we see that the population survival function $\E\,\exp(-U\int_0^{t}\alpha(s)\,\dd s)$ tends to $1 - \pi > 0$ as $t\to \infty$, which means that it is improper. In the cure model literature $\pi$ is often called the {\it incidence} part of the model, while the hazard rate and related (survival) quantities are said to belong to the {\it latency} part of the model. We adopt this terminology in the following. Usually, both parts of the model are made to depend on covariates, the most common choice being a logistic specification for the incidence part, and a Cox regression model for the latency part, that is $\pi(X_i^{\prime}\gamma) = 1/\{1 + \exp(-X_i^{\prime}\gamma)\}$ and $\alpha_{i}(t) = \alpha_0(t) \exp(Z_i^{\prime}\beta)$, in terms of the covariate vectors $X_i$ and $Z_i$. The unknown parameters of this model are $\beta$, $\gamma$, and $A_0(t) = \int_0^t\alpha_0(s)\,\dd s$, and these are estimated from the data. Methods for estimating the parameters $\beta$, $\gamma$, and $A_0(t)$ were developed by~\citet{sy2000estimation}, and by~\citet{peng2000nonparametric}; while the asymptotic properties of these estimators were studied by~\citet{fang2005maximum} and by~\citet{lu2008maximum}, building on the work of~\citet{murphy1994consistency,murphy1995asymptotic} for the gamma frailty model. 

In this paper we introduce a mixture cure model with a linear hazard rate regression model for the susceptible fraction of the population, and a logistic regression model for the incidence part. This means that in our model the $i${'}th individual has survival function  
\begin{equation}
S(t,X_i,Z_i) = 1 - \pi(X_i^{\prime}\gamma) + \pi(X_i^{\prime}\gamma)\exp(-Z_i^{\prime}\int_0^t\beta(s)\,\dd s),
\notag
\end{equation}
where
\begin{equation}
\pi(X_i^{\prime}\gamma) = \frac{\exp(X_i^{\prime}\gamma)}{1 + \exp(X_i^{\prime}\gamma)},\quad\text{and}\quad 
Z_i^{\prime}\beta(t) =  \beta_0(t) + \beta_1(t)Z_{i,1} + \cdots + Z_{i,q-1}\beta_{q-1}(t).
\label{eq::linhazard1}
\end{equation} 
We assume that $Z^{\prime}\beta(t) > 0$ for all $Z$ in the support of the covariate distribution; the covariate vectors $(X_{i,1},\ldots,X_{i,p-1})^{\prime}$ and $(Z_{i,1},\ldots,Z_{i,q-1})^{\prime}$ may be completely different, partly overlapping, or the same; the $\gamma_0,\ldots,\gamma_{p-1}$ are unknown parameters, and the $\beta_{0}(t),\ldots,\beta_{q-1}(t)$ are unknown functions. For $l = 0,\ldots,q-1$, write $B_{l}(t) = \int_0^t\beta_{l}(s)\,\dd s$ for the cumulative regression coefficients, and $B(t) = (B_0(t),\ldots,B_{q-1}(t))^{\prime}$, so that $Z_i^{\prime}B(t)$ is the cumulative hazard of the $i${'}th individual. The full parameter vector, denoted $\varphi$, is 
\begin{equation}
\varphi = (B^{\prime},\gamma^{\prime})^{\prime} = (B_0,\ldots,B_{q-1},\gamma_0,\ldots,\gamma_{p-1})^{\prime}.
\notag
\end{equation}
In the standard survival analysis setting, that is, the no cured fraction $\pi\equiv 1$ case, counting process models with hazard rates of the form $Z_i^{\prime}\beta(t) = \beta_0(t) + \sum_{l=1}^{q-1}Z_{i,l}\beta_l(t)$ were first introduced and studied by \citet{aalen1980model,aalen1989linear,aalen1993further}. See the monograph \citet[Ch.~VII.4, p.~562]{andersen1993statistical} for a general discussion of this model, and \citet{huffer1991weighted}, \citet{mckeague1994partly}, \citet{lin1994semiparametric}, and \citet{sinha2009empirical}, for important extensions and variations of the linear hazard regression model in the standard survival setting (that is, $\pi \equiv 1$).

This technical report proceeds as follows. In Section~\ref{sec::detour} we make a brief detour by the estimator introduced by~\citet{aalen1980model,aalen1989linear}. Appendix~\ref{app::aalens_linear_pfs} contains proofs of consistency and limiting normality of this estimator. These results are not new, of course, but the versions of the proof that we provide are instructive, and shed light on some of the challenges encountered, as well as techniques employed, when proving something similar for the cure model. In addition, the form of the Aalen{'}s linear hazard estimator provides the motivation for the estimation strategy that we adopt in the cure model setting. Section~\ref{sec::2.2} introduces this estimation strategy. In Section~\ref{sec::true_model_and_approx} we state our assumptions, and introduce a parametric approximation to the true semiparametric model. In Section~\ref{sec::a_locally_constant_model} we work under the assumption that this parametric approximation is indeed the true model (i.e.~the model generating the data) and derive consistency and limiting normality of our estimators under this assumption. In Section~\ref{sec::shifting_back} we use contiguity techniques to switch back from the parametric models to the semiparametric model, thus obtaining asymptotic results for semiparametric estimators under a semiparametric data generating mechanism. 

The reason for taking the route by a sequence of parametric models models and contiguity, rather than doing the computations directly under the semiparametric model, is that it eases some of the proofs, the consistency proof in particular. A by-product of this approach is that one is spared some approximations in some of the arguments.        

\section{Linear hazard regression models}\label{sec::detour}
In Section~\ref{sec::2.1} we give a brief introduction to Aalen{'}s linear hazard rate regression model and the estimator proposed by~\citet{aalen1980model,aalen1989linear}. In particular, we emphasise that this estimator lends itself naturally to an asymptotic analysis based on martingale theory. The form of Aalen{'}s estimator provides the motivation for the estimator of $\varphi$ that we propose in this technical report, and in Section~\ref{sec::2.2} we explain how. From one section to the other, we reuse many of the symbols, $N$, $Y$, and $M$, because the are so ingrained, the reader should be aware, though, that they are counting processes, at-risk processes, and martingales, associated with different models.     

\subsection{The classical setting}\label{sec::2.1}
Suppose that we are in a setting where all the individuals under study are susceptible to the event of interest. We refer to this as the {`}classical{'} or {`}standard{`} survival analysis setup. The data take the form $(T_1,\delta_1,Z_1),\ldots, (T_n,\delta_n,Z_n)$ in terms of observed lifetimes $T_i = T_i^*\wedge C_i$, being the minimum of a true lifetime $T_i^*$ and a censoring time $C_i$; the censoring indicators $\delta_1,\ldots,\delta_n$ take the value $1$ if $T_i = T_i^*$, and zero otherwise; and $Z_i$ is a $q$-dimensional vector of covariates. Under Aalen{'}s linear hazard rate regression model, the lifetimes $T_i^{*}\mid Z_i$ stem from a distribution with hazard rate given by      
\begin{equation}
Z_i^{\prime}\beta^{\circ}(t)
= \beta_{0}^{\circ}(t) + Z_{i,1}\beta_{1}^{\circ}(t) + \cdots + Z_{i,q-1}\beta_{q-1}^{\circ}(t),
\notag
\end{equation}
where the $\circ$-superscript indicates that these are the true values of the parameters. Based on the vector $(T_i,\delta_i,Z_i)$ we form the counting processes $N_i(t)$ and at-risk processes $Y_i(t)$ defined by 
\begin{equation}
N_i(t) = I\{T_i \leq t ,\delta_i = 1\},\quad \text{and}\quad Y_i(t) = I\{T_i \geq t\}, \quad\text{for $i = 1,\ldots,n$}.
\label{eq::counting_and_atrisk}
\end{equation}  
We also need the martingales 
\begin{equation}
M_i(t) = N_i(t) - \int_0^t Y_i(s)Z_i^{\prime}\beta^{\circ}(s)\,\dd s, \quad\text{for $i = 1,\ldots,n$}, 
\notag
\end{equation}
which are orthogonal and square integrable with predictable quadratic variation, $\langle M_i,M_i\rangle_{t} = \int_0^t Y_i(s) Z_i^{\prime}\beta^{\circ}(s)\,\dd s$. 
To estimate the cumulative $B^{\circ}(t)$, \citet{aalen1980model,aalen1989linear} introduced the estimator whose increments are given by 
\begin{equation}
\dd \widetilde{B}(s) = \widetilde{G}_n(s)^{-1}n^{-1}\sum_{i=1}^n Z_i \,\dd N_i(s), \quad \text{with}\quad
\widetilde{G}_n(s) = n^{-1}\sum_{i=1}^n Y_i(s)Z_i Z_i^{\prime}.
\label{eq::standard.est}
\end{equation}
The relation $\dd \{\widetilde{B}(s) - B^{\circ}(s)\} = n^{-1}\widetilde{G}_n(s)^{-1} \sum_{i=1}^n Z_i \,\dd M_i(s)$ shows that this estimator is amenable to martingale theory. In particular, $n^{-1/2}\sum_{i=1}^n \int_0^t Z_i \,\dd M_i(s)$ is a martingale with variance process $H_n(t) = n^{-1}\sum_{i=1}^n\int_0^tY_i(s) Z_iZ_i^{\prime}Z_i^{\prime}\beta^{\circ}(s) \,\dd s$. Provided $\widetilde{G}_n(s) \to_p G(s)$ and $H_n(s) \to_p H(s)$ as $n\to \infty$, where $G(s)$ and $H(s)$ are full-rank $q\times q$ matrix functions, we have process convergence of $\sqrt{n}(\widetilde{B}- B)$ to an independent increments Gaussian martingale $\Uscr$, whose variance process is 
\begin{equation}
\langle \Uscr,\Uscr\rangle_t = \int_0^t G(s)^{-1}\,\dd H(s) G(s)^{-1}.
\label{eq::aalen_estimator_clt}
\end{equation}
Three proofs of this result can be found in Appendix~\ref{app::aalens_linear_pfs}, and all three are different from what might be called the standard martingale based proof, see for example \citet[Theorem~VII.4.1, p.~575]{andersen1993statistical}. These proofs are of some interest in themselves, and they shed light on some of the challenges with deriving, but also some of the techniques used to derive, a similar result for linear hazard cure model.




\subsection{A cured fraction is present}\label{sec::2.2}
The estimator we introduce below is the natural generalisation of $\widetilde{B}$ in~\eqref{eq::standard.est} to the cure model setting. To understand how, we must delve a little deeper into the cure model. As touched upon in the introduction, when working with the cure model we split the covariates in two, the $p$-dimensional vector $X_i$, and the $q$-dimensional vector $Z_i$. This is done to distinguish those covariates affecting the probability $\pi(X_i^{\prime}\gamma^{\circ})$ of being susceptible, and those working on the hazard rate $Z_i^{\prime}\beta^{\circ}(t) = \beta_0^{\circ}(t)+ \sum_{l=1}^{q-1} Z_{i,l}\beta_l^{\circ}(t)$. 
The counting processes and the at-risk processes are defined as in~\eqref{eq::counting_and_atrisk}. With respect to the history of the observable quantities $(N_1(t),Y_1(t),X_1,Z_1),\ldots,(N_n(t),Y_n(t),X_n,Z_n)$, the compensator of the $i${'}th counting process is (see~\citet{nielsen1992counting} for details)
\begin{equation}
\int_0^tY_i(s)w_i(s,\varphi^{\circ})Z_i^{\prime}\beta^{\circ}(s)\,\dd s,
\notag
\end{equation}
where the functions $w_i(t,\varphi)$ are given by   
\begin{equation}
w_i(t,\varphi) = \frac{\pi(X_i^{\prime}\gamma)\exp(-Z_i^{\prime}\int_0^t\beta(s)\,\dd s)}{1 - \pi(X_i^{\prime}\gamma) + \pi(X_i^{\prime}\gamma)\exp(-Z_i^{\prime}\int_0^t\beta(s)\,\dd s)}.
\notag
\end{equation}
Importantly, note that when $\pi(v) = \exp(v)/\{1 + \exp(v)\}$ is a logistic function, which it will be throughout the paper, the $w_i(t,\varphi)$ functions are also logistic functions, that is
\begin{equation}
w_i(t,\varphi) = \pi(X_i^{\prime}\gamma - Z_i^{\prime}\int_0^t\beta(s)\,\dd s)
= \frac{\exp(X_i^{\prime}\gamma - Z_i^{\prime}\int_0^t\beta(s)\,\dd s)}{1 + \exp(X_i^{\prime}\gamma - Z_i^{\prime}\int_0^t\beta(s)\,\dd s)} .
\label{eq::wi.functions}
\end{equation} 
This means that with respect to the filtration of observables, the 
\begin{equation}
M_i(s) = N_i(s) - \int_0^t Y_i(s) w_i(s,\varphi^{\circ}) Z_i^{\prime}\beta^{\circ}(s)\,\dd s,\quad\text{for $i = 1,\ldots,n$},
\notag
\end{equation} 
are orthogonal, square integrable martingales. Comparing $M_i(s)$ here with the martingale in~\eqref{eq::standard.est}, it appears that a natural estimator, say $\widehat{B}$, of $B^{\circ}$ in the cure model setting is characterised by 
\begin{equation}
\begin{split}
\dd\widehat{B}(t) &= G_{n}(t,\widehat{\varphi})^{-1}n^{-1}\sum_{i=1}^n Z_i\,\dd N_i(t),\\ 
G_{n}(t,\varphi) & = n^{-1}\sum_{i=1}^n Z_iZ_i^{\prime}Y_i(t)\{\delta_i + (1 - \delta_i)w_i(T_i,\varphi)\},
\end{split}
\label{eq::B_hat_nonpara}
\end{equation}
where $\widehat{\varphi}$ is an estimate of the true $\varphi^{\circ}$. The weights $\delta_i + (1 - \delta_i)w_i(T_i,\varphi)$ on the at-risk indicators $Y_i(t)$ are intuitively appealing because $\delta_i = 1$ implies that $U_i = 1$, so if $\delta_i = 1$, the at-risk status of the $i${'}th individual at time $t$ is indeed $Y_i(t)$. If $\delta_i = 0$, on the other hand, the value of $U_i$ is unknown, $Y_i(t)$ is equal to $1$ for all $t$, and $w_i(T_i,\varphi)$ is the probability we subscribe to the $i${'}th individual being at risk. 

Contrary to what is the case for the Aalen estimator in~\eqref{eq::standard.est}, however, the difference 
\begin{equation}
\widehat{B}(t) - B^{\circ}(t), 
\notag
\end{equation} 
is not a martingale, hence the analytical tractability of the Aalen estimator provided by the availability of martingale theory does not immediately carry over to the estimator in~\eqref{eq::B_hat_nonpara}. There are two reasons for still studying the estimator in~\eqref{eq::B_hat_nonpara}: First, it is computationally extremely easy to compute, and, second, it is {\it almost} a martingale.

\section{The true model and an approximation}\label{sec::true_model_and_approx}
We have independent and identically distributed (i.i.d.)~replicates $(T_i,\delta_i,X_i,Z_i)$ for $i = 1,\ldots,n$ of $(T,\delta,X,Z)$, observed over a finite interval of time $[0,\tau]$, where $X$ and $Z$ are $p$- and $q$-dimensional column vectors of covariates, respectively, with the first elements being a $1$ corresponding to an intercept. The observed time $T$ is the minimum of a true survival time $T^{*}$ and a censoring time $C$, that is $T = T^*\wedge C$, and $\delta$ is an indicator taking the value $1$ if an event is observed, zero otherwise. The censoring times are assumed to be independent draws from an absolutely continuous distribution $H_c$ on $(0,\tau]$, independent of both the $X$- and $Z$-covariates, and of the survival times. The survival times $T_i^{*}\mid Z_i$ stem from a mixture distribution with hazard rates 
\begin{equation}
U_iZ_i^{\prime}\beta^{\circ}(t)
= U_i\{\beta_{0}^{\circ}(t) + Z_{i,1}\beta_{1}^{\circ}(t) + \cdots + Z_{i,q-1}\beta_{q-1}^{\circ}(t)\},\quad\text{for $i = 1,\ldots,n$},
\notag
\end{equation}
where $U_1,\ldots,U_n$ are independent Bernoulli random variables whose means conditional on $X$ are 
\begin{equation}
\pi_i = \pi(X_i^{\prime}\gamma^{\circ}) = \frac{\exp(-X_i^{\prime}\gamma^{\circ})}{1 + \exp(-X_i^{\prime}\gamma^{\circ})},\quad \text{for $i=1,\ldots,n$}. 
\notag
\end{equation}
As above, the $\circ$-superscript denotes the true values of the parameters. The function denoted $\pi(v) = \exp(v)/\{1 + \exp(v)\}$ will always be the logistic function, and we use the shorthand $\pi_i = \pi(X_i^{\prime}\gamma)$ for the Bernoulli mean of the $i${'}th individual. The true survival function of the $i${'}th individual is then
\begin{equation}
S(t ; X_i,Z_i,\varphi^{\circ}) = 1 - \pi(X_i^{\prime}\gamma^{\circ}) + \pi(X_i^{\prime}\gamma^{\circ})\exp\{-Z_i^{\prime}B^{\circ}(t)\}.
\label{eq::popsurv1}
\end{equation} 
In addition to the distributional assumptions already mentioned, we impose the following conditions: 
\begin{assumption}\label{assumption1} The functions $\beta_0^{\circ}(s),\ldots,\beta_{q-1}^{\circ}(s)$ are two times continuously differentiable on $[0,\tau]$.
\end{assumption}
\begin{assumption}\label{assumption2} The parameter $\gamma^{\circ} = (\gamma_0^{\circ},\ldots,\gamma_{p-1}^{\circ})^{\prime}$ lies in the interior of a set $\Theta_{\gamma}\subset \real^p$, and $\Theta_{\gamma}$ is closed and bounded.  
\end{assumption}
\begin{assumption}\label{assumption3} The covariates $X$ and $Z$ are linearly independent; bounded with probability one; and the matrices $\E\, ZZ^{\prime}$ and $\E\, XX^{\prime}$ are both positive definite. 
\end{assumption}
\begin{assumption}\label{adhoc_assumption} Apart from the intercept terms, all the components of both covariate vectors $X$ and $Z$ stem from the same distribution.   
\end{assumption} 
  
The last assumption is only used in the proof of Lemma~\ref{lemma::uniqueness} where uniqueness of the solution to the expectation of a set of estimating equations is proved. It is likely not a necessary condition for this lemma to be true. The first derivative of $\beta_{l}^{\circ}(s)$ with respect to time is $\dot{\beta}_l^{\circ}(s)$, and $\dot{\beta}^{\circ}(s)$ is the column vector $(\dot{\beta}_0^{\circ}(s),  \ldots,\dot{\beta}_{q-1}^{\circ}(s))^{\prime}$, and similarly for the second derivative $\ddot{\beta}_l^{\circ}(s)$, and the column vector $\ddot{\beta}^{\circ}(s)$ of second derivatives. 

From now on, when expectations are taken with respect to a distribution of $T_1,\ldots,T_n$, both conditionally on the covariates and unconditionally, this will be denoted by subscripts, for example $\E_{\varphi}\,(\cdot)$; while an expectation without a subscript $\E\,(\cdot)$, means that the expectation is taken solely with respect to the distribution of the covariates. In the following we use the notation $Y(s)$ and $w(s,\theta)$ for the random variables $I\{T \geq s\}$ and $\pi\{X^{\prime}\gamma - Z^{\prime}B(s)\}$, respectively; and write 
\begin{equation}
y(s;X,Z) = \E_{\varphi}\,\{Y(s)\mid X,Z\} = \pi(X^{\prime}\gamma)\{1-H_c(s)\}\exp\{-Z^{\prime}B(s)\} + 1 - \pi(X^{\prime}\gamma),
\notag
\end{equation}
for the conditional expectation of $Y(s)$, and $y(s)$ for the marginal $y(s) = \E\, y(s;X,Z)$. It will be clear from the context what distribution the expectations in $y(s)$ and $y(s;X,Z)$ are with respect to. Notice also that $y(s)$ and $y(s;X,Z)$ are continuous in $s$, with continuous derivatives on $[0,\tau]$. Moreover, from the i.i.d.~assumption, it follows from the Glivenko--Cantelli theorem (see e.g.~\citet[Theorem~19.1, p.~266]{vanderVaart1998}), that $n^{-1}\sum_{i=1}^n Y_i(t) \to_p y(t)$, and also 
\begin{equation}
n^{-1}\sum_{i=1}^n \E_{\varphi}\,\{Y_i(t)\mid X,Z\} \overset{p}\to y(t), 
\notag
\end{equation}
both uniformly on $[0,\tau]$. The distribution of the data specified by Assumptions~\ref{assumption1}--\ref{assumption2} is denoted $P_{\circ}$. This is the distribution associated with the parameter value $\varphi^{\circ} = \{(B^{\circ})^{\prime},(\gamma^{\circ})^{\prime}\}^{\prime}$, where $B_{l}^{\circ}(t) = \int_0^{t}\beta_l^{\circ}(s)\,\dd s$ for $l = 0,\ldots,{q-1}$, and the $\beta_0^{\circ}(s),\ldots,\beta_{q-1}^{\circ}(s)$ are at least two times continuously differentiable on $[0,\tau]$. 

The approximation alluded to in the title of this section are the parametric distributions $P_{\star}^K$ constructed as follows: For some integer $K \geq 2$, let 
\begin{equation}
0 = v_0 < v_{1} < \cdots < v_{K-1} < v_K = \tau,
\notag
\end{equation} 
be a partition of $[0,\tau]$. Set $W_j = [v_{j-1},v_j)$ and assume that $v_{j} = j\tau/K$ for $j=1,\ldots,K$, and let $I_{W_j}(t) = 1$ if $t \in W_j$, and zero otherwise. For some $K$, the parameters of $P_{\star}^K$ are 
\begin{equation}
\gamma^{\star} = \gamma^{\circ},\quad \text{and}\quad B^{\star}(t) = \int_0^t \beta^{\star}(s)\,\dd s,
\notag
\end{equation}
where 
\begin{equation}
\beta^{\star}(t) = \sum_{j=1}^K \beta_j^{\star} I_{W_j}(t),\quad\text{and}\quad\beta_j^{\star} = \beta^{\circ}(v_{j-1}),\quad\text{for $j =1,\ldots,K$},
\notag
\end{equation}
are column vectors, e.g.~$\beta_j^{\star} = (\beta_{0,j}^{\star},\ldots,\beta_{q-1,j}^{\star})^{\prime}$. To not overburden an already heavy notation, we avoid using $K$ super/sub-script for the parameters of the distributions $P_{\star}^K$. The distributions $P_{\star}^K$ are fully specified by the parameter vector $\theta^{\star}$ that lives in $\Theta_K \subset \real^{qK}\times \Theta_{\gamma} \subset \real^{qK}\times \real^p$, where an arbitrary element of $\Theta_K$ is of the form
\begin{equation}
\theta = (\beta_0^{\prime},\ldots,\beta_{q-1}^{\prime},\gamma^{\prime})^{\prime}
= (\beta_{0,1},\ldots,\beta_{q-1,1},\ldots,\beta_{0,K},\ldots,\beta_{q-1,K},\gamma_0,\ldots,\gamma_{p-1})^{\prime}.
\notag
\end{equation}
Notice that by Assumptions~\ref{assumption1}--\ref{assumption2} we can, and will, assume that the parameter spaces $\Theta_K$ are all closed and bounded. With a slight abuse of notation, we index functions both by $\theta$ and by $\varphi$, even though these are parameters living in different spaces. Throughout, $\norm{\alpha} = (\sum_{\ell=1}^q \alpha_{\ell}^2)^{1/2}$ is the Euclidian norm, and $\Delta \xi(t) = \xi(t) - \xi(t-)$ is the jump of a process $\xi$ at time $t$. Here $\xi(t-) = \lim_{s \uparrow t}\xi(s-)$, with the limit taken from the left, as usual.

The strategy is now as follows: In Section~\ref{sec::a_locally_constant_model} we work under the assumption that the data stem from the distribution $P_{\star}^K$ for some fixed $K$, and study the large-sample properties of our estimators as $n\to\infty$. This section is thus purely parametric. In Section~\ref{sec::cov_shrink} we switch to a triangular array setup, where we assume that the data are generated by $P_{\star}^{K}$ for increasing values of $K$, eventually tending to infinity, and derive large-sample results when both $n\to \infty$ and $K \to \infty$. Finite dimensional convergence in distribution is denoted by ${``}\to_d{"}$, while ${``}\Rightarrow{"}$ indicates full process convergence. In Section~\ref{sec::shifting_back} we use measure change techniques, inspired by those developed in~\citet{mykland2009inference} in a high-frequency setup, to switch back to having the data generated by the distribution $P_{\circ}$ associated with the continuous regression functions $\beta_0^{\circ}(s),\ldots,\beta_{q-1}^{\circ}(s)$.          


\subsection{The locally constant model}\label{sec::a_locally_constant_model} 
In this section we assume that for a fixed partition $0 = v_0 < v_1 < \cdots v_{K-1} < v_K = \tau$, the data are generated by a $P_{\star}^K$ distribution, that is, $K$ is held fixed. The expectation taken with respect to this distribution is denoted $\E_{\theta^{\star}}(\cdot)$. This means that our model for the data is given by the survival functions   
\begin{equation}
S(t ;X,Z,\theta) = 1 - \pi(X^{\prime}\gamma) + \pi(X^{\prime}\gamma)\exp\{-Z^{\prime}\int_0^{t}\sum_{j=1}^K \beta_jI_{W_j}(s)\,\dd s\}, 
\label{eq::surv1_model}
\end{equation} 
with the true model being $S(t ;X,Z,\theta^{\star})$. The distributional assumptions on the covariates and censoring times are as stated above. In particular, for $i = 1,\ldots,n$,
\begin{equation} 
M_i(t,\theta^{\star}) = N_i(t) - \int_0^t Y_i(s) w_i(s,\theta^{\star}) Z_i^{\prime}\beta^{\star}(s)\,\dd s, 
\notag
\end{equation}
are martingales under $P_{\star}^K$. 

The natural counterpart of~\eqref{eq::B_hat_nonpara} under the model given here are the estimators $\widehat{\beta}_1,\ldots,\widehat{\beta}_K$ characterised by 
\begin{equation}
\widehat{\beta}_j = G_{n,j}(\widehat{\theta})^{-1} n^{-1}\sum_{i=1}^nZ_i\int_{W_j} \,\dd N_i(s),\quad\text{for $j=1,\ldots,K$}, 
\label{eq::betahat1}
\end{equation} 
with
\begin{equation}
G_{n,j}(\theta) = n^{-1}\sum_{i=1}^n Z_iZ_i^{\prime}r_j(T_i)\{\delta_i + (1-\delta_i)w_i(T_i,\theta)\}, \quad\text{for $j=1,\ldots,K$},
\label{eq::betahat2}
\end{equation} 
for some estimator $\widehat{\gamma}$. The functions $r_{j}(t)$ are given by
\begin{equation}
r_{j}(t) = \int_{W_j} I\{t \geq s\} \,\dd s = \int_0^{t}I_{W_j}(s)\,\dd s,\quad\text{for $j = 1,\ldots,K$}.
\label{eq::rj_funcs}
\end{equation} 
Thus, $r_{j}(T_i) = \int_{W_j} I\{T_i \geq s\} \,\dd s = \int_{W_j} Y_i(s) \,\dd s$ is the amount of time the $i${'}th individual spends in the $j${'}th time interval. Apart from the intuitive appeal of the $\widehat{\beta}_j$ in~\eqref{eq::betahat1}, these estimators can be motivated, and are in fact defined, by certain estimating equations. To see how, we start with the log-likelihood function of the model in~\eqref{eq::surv1_model}. It is
\begin{equation}
\ell_n^K(\theta)  = \sum_{i=1}^n [ \delta_i \{\log \pi_i + \log Z_i^{\prime} \beta(T_i)  - Z_i^{\prime}B(T_i) \} + (1 - \delta_i) \log \{1 - \pi_i + \pi_i \exp(- Z_i^{\prime}B(T_i)) \} ],
\notag
\end{equation} 
where the superscript $K$ indicates that this likelihood function is defined relative to a given partition $W_{1},\ldots,W_{K}$ of $[0,\tau]$. After some algebra and using the expression for the $w_i(T_i,\theta)$ given in~\eqref{eq::wi.functions}, one finds that the score functions are\footnote{For $\partial\ell_n^K(\theta) /\partial\gamma$, use that $1 - w_i(T_i,\theta) = (1 - \pi_i)/\{1 - \pi_i + \pi_ie^{-Z_iB(T_i)}\}$, then $\partial\ell^K(\theta,T_i)/\partial \gamma = \delta_i (1 - \pi_i) - (1-\delta_i)\{\pi_i(1 - \pi_i) - \pi_i(1 - \pi_i)e^{-Z_iB(T_i)}\}/\{1 - \pi_i + \pi_ie^{-Z_iB(T_i)}\}
 = \delta_i (1 - \pi_i) - (1-\delta_i)[\pi_i \{1 - w_i(T_i,\theta)\}
- (1 - \pi_i)w_i(T_i,\theta)]= \delta_i (1 - \pi_i) + (1-\delta_i)(w_i(T_i,\theta) - \pi_i)
= \delta_i + (1 - \delta_i)w_i(T_i,\theta) - \pi_i$.}
\begin{equation}
\begin{split}
\frac{\partial}{\partial\beta_j}\ell_n^K(\theta) 
& = \sum_{i=1}^n\big[ \frac{Z_i\delta_iI_{W_j}(T_i)}{Z_i^{\prime}\beta_j} 
- r_j(T_i) \{\delta_i + (1-\delta_i)w_i(T_i,\theta)\} Z_i \big],\quad\text{for $j=1,\ldots,K$},\\
\frac{\partial}{\partial\gamma}\ell_n^K(\theta) & = \sum_{i=1}^n X_i\{\delta_i + (1-\delta_i)w_i(T_i,\theta) - \pi_i\}.
\end{split}
\notag
\end{equation}
The estimators characterised by~\eqref{eq::betahat1}--\eqref{eq::betahat2} are not necessarily the zeros of these equations, and it turns out that for moderate $K$ the likelihood $\ell_{n}^K(\theta)$ is computationally very hard to maximise (it is, after all, a model with $qK + p$ parameters). Therefore, we multiply each of the elements of the sums $\partial\ell_n^K(\theta) /\partial\beta_j$ by $Z_i^{\prime}\beta_j$ for $j = 1,\ldots,K$, to obtain the vector valued function 
\begin{equation}
\Psi_{n}^K\colon \mathbb{R}^{qK + p}\to \mathbb{R}^{qK + p}, 
\notag
\end{equation}
whose vector valued elements are 
\begin{equation}
\begin{split}
&\Psi_{n,j}^K(\theta)  = \frac{1}{n}\sum_{i=1}^n Z_i[\delta_iI_{W_j}(T_i) - r_j(T_i)\{\delta_i + (1-\delta_i)w_i(T_i,\theta)\}Z_i^{\prime}\beta_j ],\quad\text{for $j=1,\ldots,K$},\\
& \Psi_{n,K+1}^K(\theta)  = \frac{1}{n}\sum_{i=1}^n X_i\{\delta_i + (1-\delta_i)w_i(T_i,\theta) - \pi_i\}. 
\end{split}
\label{eq::system1}
\end{equation}
Thus, $\Psi_n^K(\theta)$ is the $qK + p$ dimensional column vector where the $\Psi_{n,1}^K(\theta),\ldots,\Psi_{n,K+1}^K(\theta)$ are stacked on top of each other. The estimator $\widehat{\theta} = (\widehat{\beta}^{\prime},\widehat{\gamma}^{\prime})^{\prime}\in \Theta_K$ is defined as the solution to 
\begin{equation}
\Psi_n^K(\theta) = 0.
\label{eq::the_est_eq}
\end{equation}  
The estimator $\widehat{\theta}$ depends, of course, on both $n$ and $K$, but we stick to writing $\widehat{\theta}$. Given an estimator $\widehat{\gamma}$, we see that the solution $\widehat{\theta}$ to $\Psi_{n,j}^K(\theta) = 0$ for $j=1,\ldots,K$ must have $\widehat{\beta}_j$ as defined in~\eqref{eq::betahat1}--\eqref{eq::betahat2}. It should be noted that the estimating equations leading to the Aalen estimator $\widetilde{B}$ of~\eqref{eq::standard.est} can be derived in the same manner, that is, by way of the score functions. This is done explicitly in Appendix~\ref{app::aalens_linear_pfs}, and also in~\citet{mckeague1994partly}, in both cases leading to estimating equations amenable to martingale theory. Note also that we can rewrite the equations in~\eqref{eq::system1} as (see Appendix~\ref{app::some_results}),
\begin{equation}
\begin{split}
& \Psi_{n,j}^K(\theta)  = \frac{1}{n}\sum_{i=1}^n Z_i\int_0^{\tau} 
\big\{I_{W_j}(s) - (1-w_i(s,\theta))r_j(s) \, Z_i^{\prime}\beta_j \big\} \,\dd M_{i}(s,\theta),\quad j = 1,\ldots,K,\\
& \Psi_{n,K+1}^K(\theta)  =\frac{1}{n}\sum_{i=1}^n X_i\int_0^{\tau}\{1-w_i(s,\theta)\}\,\dd M_{i}(s,\theta),
\end{split}
\label{eq::system2}
\end{equation}  
where $\dd M_i(s,\theta)$ is a shorthand for $\dd N_i(s) - Y_i(s)w_i(\theta,s)Z_i^{\prime}\beta(s)\,\dd s$, that is, when evaluated in $\theta^{\star}$ the $\dd M_i(s,\theta^{\star})$ are martingale increments under $P_{\star}^K$. Consequently, $\Psi_n^K(\theta^{\star})$ is a martingale under $P_{\star}^K$, and $\Psi_n^K(\theta) = 0$ is an unbiased estimating equation. 

\subsection{Parametric large-sample results}
We now proceed to the asymptotics of $\widehat{\theta}$ as $n$ tends to infinity (and $K$ is held constant), working under a fixed distribution $P_{\star}^K$ for the data. Let $h(t,x,z,\theta) \in \real^{qK+p}$ be the function
\begin{equation}
h(t,x,z,\theta) = 
\begin{pmatrix}
h_1(t,x,z,\theta)\\
\vdots \\
h_K(t,x,z,\theta)\\
h_{K+1}(t,x,z,\theta)
\end{pmatrix}
= \begin{pmatrix}
zg_1(s,x,z,\theta)\\ 
\vdots \\
zg_K(s,x,z,\theta)\\ 
x\{1 - w(s,\theta)\}
\end{pmatrix},
\label{eq::h_func}
\end{equation}
where
\begin{equation}
\begin{split}
g_j(s,x,z,\theta) & = I_{W_j}(s) - \{1 - w(s,\theta)\}r_j(s)z^{\prime}\beta_j,\quad\text{for $j = 1,\ldots,K$}.
\end{split}
\label{eq::small_g_func}
\end{equation}
Denote by $\Psi_{\star}^K$ the probability limit of $\Psi_n^K$ under the $P_{\star}^K$-distribution when $n \to \infty$, that is $\Psi_{\star}^K(\theta) = \E_{\theta^{\star}} \,\Psi_{n}^K(\theta)$, where
\begin{equation}
\Psi_{\star}^K(\theta) = \E\,\int_0^{\tau} h(s,X,Z,\theta) y(s;X,Z)\{w(s,\theta^{\star})Z^{\prime}\beta^{\star}(s) - w(s,\theta)Z^{\prime}\beta(s)\}\,\dd s.
\label{eq::Psi_avg1}
\end{equation}
Clearly, $\theta^{\star}$ is a solution to the equation $\Psi_{\star}^K(\theta) = 0$. The next lemma shows that, in $\Theta_K$, it is the only one. 

\begin{lemma}\label{lemma::uniqueness} The parameter value $\theta^{\star}$ is the unique solution to $\Psi_{\star}^K(\theta) = 0$.   
\end{lemma}
\begin{proof} See Appendix~\ref{app::lemma::uniqueness}. 
\end{proof}

\begin{lemma}\label{lemma::consistency1} The sequence of solutions $\widehat{\theta}$ to $\Psi_n^K(\theta) = 0$ is consistent for $\theta^{\star}$.  
\end{lemma}
\begin{proof} We have that 
\begin{equation}
\begin{split}
\Psi_n^K(\theta) & = \frac{1}{n}\sum_{i=1}^n\int_0^{\tau}h(s,X_i,Z_i,\theta) \,\dd M_i(s,\theta)\\
& = \frac{1}{n}\sum_{i=1}^n\int_0^{\tau}h(s,X_i,Z_i,\theta) \big[\dd M_i(s,\theta^{\star}) 
+ Y_i(s)Z_i^{\prime}\{w_i(s,\theta^{\star})\beta^{\star}(s) - w_i(s,\theta) \beta(s)\}\,\dd s \big] \\
& = \frac{1}{n}\sum_{i=1}^n\int_0^{\tau}h(s,X_i,Z_i,\theta) \,\dd M_i(s,\theta^{\star}) + \Psi_{\star}^K(\theta) + o_p(1),
\end{split} 
\notag
\end{equation} 
uniformly in $\theta$ by Assumptions~\ref{assumption1}--\ref{assumption2}, that is, using that $\Theta_K$ is compact, that the summands are continuous in $\theta$ for every data point, and that they are dominated by an integrable function \citep[Theorem~16(a), p.~108]{ferguson1996course}. By Assumptions~\ref{assumption1}--\ref{assumption2}, the same applies to the function $\theta \mapsto \int_0^{\tau} h(s,x,z,\theta)\,\dd M(s,\theta^{\star})$, therefore,
\begin{equation}
\sup_{\theta \in \Theta_K}\norm{\Psi_n^K(\theta)  - \Psi_{\star}^K(\theta) } = 
\sup_{\theta \in \Theta_K}\norm{ \frac{1}{n}\sum_{i=1}^n\int_0^{\tau}h(s,X_i,Z_i,\theta) \,\dd M_i(s,\theta^{\star})} + o_p(1) \overset{p}\to 0,
\notag
\end{equation}
as $n \to \infty$. Coupled with the uniqueness of $\theta^{\star}$ from Lemma~\ref{lemma::uniqueness}, Theorem~5.9 in~\citet[p.~46]{vanderVaart1998} gives the result.   
\end{proof}
From the expression for $\Psi^K_n(\theta)$ in the proof of Lemma~\ref{lemma::consistency1}, we see that, when it is evaluated in $\theta^{\star}$,  
\begin{equation}
\Psi_n^K(\theta^{\star}) = \frac{1}{n}\sum_{i=1}^n \int_0^{\tau}h(s,X_i,Z_i,\theta^{\star}) \,\dd M_i(s,\theta^{\star}) + o_p(1), 
\notag
\end{equation}
which gives the following expression for the approximate variance of $\Psi_n(\theta^{\star})$, namely,
\begin{equation}
\Gamma^K_{\theta^{\star}} = \E_{\theta^{\star}}\,\int_0^{\tau}h(s,x,z,\theta^{\star})h(s,x,z,\theta^{\star})^{\prime} y(s;x,z) w(s,\theta^{\star})z^{\prime}\beta^{\star}(s)\,\dd s.
\notag
\end{equation}

Let $\dot{\Psi}_{\theta}^K$ be the derivative of $-\Psi_{\star}^K(\theta)$ with respect to $\theta$. In the remainder of the paper, we assume that this matrix is invertible for all $K$. This assumption likely follows directly from Assumption~\ref{assumption3}, but since we have yet to show it, we state it as an assumption. 

\begin{assumption}\label{assumption_invertible} For all $K$, the matrices $\dot{\Psi}_{\theta^{\star}}^K$ are invertible. 
\end{assumption}  

\begin{prop}\label{prop::simple_clt1} For some fixed $K$, 
\begin{equation}
\sqrt{n}(\widehat{\theta} - \theta^{\star}) = (\dot{\Psi}_{\theta^{\star}}^K)^{-1} \sqrt{n}\Psi_n^K(\theta^{\star}) + o_p(1) \to_d \normal_{qK+p}\{0, (\dot{\Psi}_{\theta^{\star}}^K)^{-1}\Gamma_{\theta^{\star}}^K((\dot{\Psi}_{\theta^{\star}}^K)^{-1})^{\prime}\},
\notag
\end{equation} 
under $P_{\star}^K$ as $n \to \infty$. 
\end{prop}
\begin{proof} By~\citet[Theorem~5.21, p.~52]{vanderVaart1998} it suffices to show the terms in the sum $\Psi_n^K(\theta)$, that is $\int_0^{\tau}h(s,X_i,Z_i,\theta)\,\dd M_i(s,\theta)$, are Lipschitz, and that $\dot{\Psi}_{\theta^{\star}}^K$ is invertible. From Appendix~\ref{app::Psi_dot}, we see that in the matrix $\dot{\Psi}_{\theta}^K$ the parameters enter either linearly or through functions that are bounded by $1$, so $\sup_{\theta \in \Theta_K}\norm{\dot{\Psi}_{\theta}^K} < \infty$ by Assumptions~\ref{assumption1}--\ref{assumption2}, and the Lipschitz condition follows. 
\end{proof}

Recall that $r_{j}(s) = \int_{W_j}I\{s \geq u\}\,\dd u$ as defined in~\eqref{eq::rj_funcs}, and introduce the $(q + p) \times (qK+p)$ matrix $H_t$ that for $\theta \in \Theta_K$ is such that 
\begin{equation}
H_t \theta = (B_0(t),\ldots,B_{q-1}(t),\gamma_0,\ldots,\gamma_{p-1})^{\prime}.
\notag
\end{equation} 
As an example, suppose $p =2$, $q = 2$, and $K = 3$. In this case 
\begin{equation}
\theta = (\beta_{0,1},\beta_{1,1},\beta_{0,2},\beta_{1,2},\beta_{0,3},\beta_{1,3},\gamma_0,\gamma_1)^{\prime}, 
\notag
\end{equation}
and   
\begin{equation}
H_{t} =  
\begin{pmatrix}
r_1(t) & 0 & r_2(t) & 0 & r_3(t) & 0 & 0 & 0\\
0 & r_1(t) & 0 &r_2(t) & 0 & r_3(t) & 0 & 0 \\
0 & 0 & 0 & 0 & 0 & 0 & 1 & 0\\
0 & 0 & 0 & 0 & 0 & 0 & 0 & 1
\end{pmatrix}.
\notag
\end{equation}
\begin{corollary}\label{corollary::simple_clt2} For fixed $K$ and some fixed $0 <t \leq \tau$ 
\begin{equation}
\sqrt{n}H_t(\widehat{\theta} - \theta^{\star}) = 
\sqrt{n}
\begin{pmatrix}
\widehat{B}(t) - B^{\star}(t)\\
\widehat{\gamma} - \gamma^{\star}
\end{pmatrix}
\overset{d}\to \normal_{q+p}\{0, H_t(\dot{\Psi}_{\theta^{\star}}^K)^{-1}\Gamma_{\theta^{\star}}^K((\dot{\Psi}_{\theta^{\star}}^K)^{-1})^{\prime}H_t^{\prime}\},
\label{eq::simple_clt2}
 \end{equation}
under $P_{\star}^K$ as $n\to \infty$. 
\end{corollary}
\begin{proof} This is direct from Proposition~\ref{prop::simple_clt1} via an application of the delta-method.
\end{proof}
The next lemma presents an approximation that is key to what follows. Recall from~\eqref{eq::betahat2} that $G_{n,j}(\theta) = n^{-1}\sum_{i=1}^n Z_iZ_i^{\prime}r_j(T_i)\{\delta_i + (1-\delta_i)w_i(T_i,\theta)\}$ for $j = 1,\ldots,K$. For each $j$, define 
\begin{equation}
\begin{split}
G_j(\theta) &= \E\, ZZ^{\prime}\int_{W_j} y(s;X,Z)w(s,\theta)\,\dd s\\ 
& \quad +\E\, ZZ^{\prime}\int_0^{\tau} r_j(s)\{1 - w(s,\theta)\} y(s;X,Z) \{w(s,\theta^{\star})Z^{\prime}\beta^{\star}(s) - w(s,\theta)Z^{\prime}\beta(s)\}\,\dd s.
\end{split}
\label{eq::Gj_approx}
\end{equation}
\begin{lemma}\label{lemma::Gapprox} For $j = 1,\ldots,K$, $G_{n,j}(\theta) = G_j(\theta) + o_p(1)$, as $n \to \infty$, uniformly in $\theta$.   
\end{lemma}
\begin{proof} Write
\begin{equation}
\begin{split}
G_{n,j}(\theta) & = \frac{1}{n}\sum_{i=1}^nZ_iZ_i^{\prime}\big[\int_{W_j}Y_i(s)w_i(s,\theta)\,\dd s
+ \int_0^{\tau} r_{j}(s) \{1 - w_i(s,\theta)\}\,\dd M_i(s,\theta^{\star}) \\
& \qquad \qquad \qquad \qquad + \int_0^{\tau} r_{j}(s) \{1 - w_i(s,\theta)\}Y_i(s)Z_i^{\prime}\{w_i(s,\theta^{\star})\beta^{\star}(s) - w_i(s,\theta)\beta(s)\}\,\dd s\big].
\end{split}
\notag
\end{equation} 
and note that $|r_j(s)| \leq |v_{j} - v_{j-1}|$. The claim follows from the same arguments used to prove Lemma~\ref{lemma::consistency1}.
\end{proof}
Note also that since $\theta \mapsto G_{n,j}(\theta)$ is continuous and $\widehat{\theta}$ is consistent for $\theta^{\star}$,
\begin{equation}
\begin{split}
G_{n,j}(\widehat{\theta}) & = G_{n,j}(\theta^{\star}) + o_p(1),\quad \text{as $n \to \infty$},\\
\end{split}
\notag
\end{equation}
and we get from Lemma~\ref{lemma::Gapprox} that $G_{n,j}(\theta^{\star})  = G_j(\theta^{\star}) + o_p(1) = \E\, ZZ^{\prime} \int_{W_j}y(s;X,Z)w(s,\theta^{\star})\,\dd s + o_p(1)$, where $G_j(\theta)$ is defined in~\eqref{eq::Gj_approx} (see \citet[p.~108]{ferguson1996course}). Also, because $\theta^{\star}$ is the probability limit of $\widehat{\theta}$, we can use the expression for $\widehat{\beta}_j$ in~\eqref{eq::betahat1}, continuous mapping, and the Cram{\'e}r--Slutsky rules to see that
\begin{equation}
\begin{split}
\widehat{\beta}_j & = G_{n,j}(\widehat{\theta})^{-1}\frac{1}{n}\sum_{i=1}^n Z_i \int_{W_j} \dd N_i(s)\\
& = G_{n,j}(\widehat{\theta})^{-1}\frac{1}{n}\sum_{i=1}^n Z_i \{\int_{W_j} \dd M_i(s,\theta^{\star}) + Y_i(s) w(s,\theta^{\star})Z_i^{\prime}\beta^{\star}_j \,\dd s\}\\
& = G_{j}(\theta^{\star})^{-1}\frac{1}{n}\sum_{i=1}^n Z_i \int_{W_j} Y_i(s) w(s,\theta^{\star})Z_i^{\prime}\beta^{\star}_j \,\dd s + o_p(1) = \beta_j^{\star} + o_p(1), \quad\text{for $j = 1,\ldots,K$}.
\end{split}
\notag
\end{equation}
From this expression it is immediate that 
\begin{equation}
\begin{split}
\widehat{\beta}_j - \beta_j^{\star} & = G_{j}(\theta^{\star})^{-1}\frac{1}{n}\sum_{i=1}^n Z_i\int_{W_j}\,\dd M_i(s) + o_p(1), \quad\text{for $j = 1,\ldots,K$},
\end{split}
\label{eq::beta_star_nice}
\end{equation}
as $n \to \infty$. This expression mimics the exact martingale expression for the corresponding difference in the standard survival case (i.e.~$\pi \equiv 1$), as explored in Appendix~\ref{app::aalens_linear_pfs}, and is key to the theory of the next section.

\section{Shrinking intervals}\label{sec::shrink}
In this section we study the result $H_t \sqrt{n}(\widehat{\theta} - \theta^{\star}) \to_d \normal_{q+p}\{0, H_t(\dot{\Psi}_{\theta^{\star}}^K)^{-1}\Gamma_{\theta^{\star}}^K((\dot{\Psi}_{\theta^{\star}}^K)^{-1})^{\prime}H_t^{\prime}\}$ from~\eqref{eq::simple_clt2} when the mesh size tends to zero, that is, when $K \to \infty$, with the aim of arriving at a process convergence result for our estimators. For a given partition, $\Delta_j = v_{j} - v_{j-1}$ for $j = 1,\ldots,K$, so that $K \to \infty$ means that $\max_{j \leq K}\Delta_j \to 0$. We also assume that $v_{j} = j\tau/K$ for all $j$. In a first part we work under the sequence $(P_{\star}^K)_{K}$ of distributions. Subsequently, we adjust back to assuming that the data stem from the distribution $P_{\circ}$ associated with $\gamma^{\circ}$ and the continuous regression coefficients $\beta_0^{\circ}(t),\ldots,\beta_{q-1}^{\circ}(t)$, as defined in Assumptions~\ref{assumption1}--\ref{assumption2}.  

\subsection{Shrinking intervals and triangular arrays}\label{sec::cov_shrink}
For the matrices introduced in Proposition~\ref{prop::simple_clt1}, write 
\begin{equation}
\dot{\Psi}_{\theta}^K 
= 
\begin{pmatrix}
\dot{\Psi}_{\theta,00}^K & \dot{\Psi}_{\theta,01}^K\\
\dot{\Psi}_{\theta,10}^K & \dot{\Psi}_{\theta,11}^K
\end{pmatrix},
\quad \text{and}\quad 
(\dot{\Psi}_{\theta}^K)^{-1} 
= 
\begin{pmatrix}
\dot{\Psi}_{\theta}^{K,00} & \dot{\Psi}_{\theta}^{K,01}\\
\dot{\Psi}_{\theta}^{K,10} & \dot{\Psi}_{\theta}^{K,11}
\end{pmatrix},
\label{eq::dotPsi}
\end{equation}
where $\dot{\Psi}_{\theta,00}^K$ is of dimension $qK\times qK$; $\dot{\Psi}_{\theta,01}^K$ is $qK\times p$; $\dot{\Psi}_{\theta,10}^K$ is $p\times qK$; and $\dot{\Psi}_{\theta,11}^K$ is $p\times p$, with the same dimensions for the blocks of the inverse. See Appendix~\ref{app::Psi_dot} for more detail on these two matrices. Denote by $h_{1:K}(s,x,z,\theta)$ the vector containing the first $qK$ elements of the vector $h(s,x,z,\theta)$ introduced in~\eqref{eq::h_func}, that is 
\begin{equation}
h_{1:K}(s,x,z,\theta) =
\begin{pmatrix}
h_1(s,x,z,\theta)\\
\vdots\\
h_K(s,x,z,\theta)
\end{pmatrix}
= 
\begin{pmatrix}
zg_1(s,x,z,\theta)\\
\vdots\\
zg_K(s,x,z,\theta)
\end{pmatrix}.
\notag
\end{equation}
Define the sequence of processes
\begin{equation} 
\Uscr_n^{K} = \sqrt{n}\{(\widehat{B} - B^{\star})^{\prime},(\widehat{\gamma} - \gamma^{\star})^{\prime}\}^{\prime},
\label{eq::Z_n_K}
\end{equation} 
so that $\Uscr_n^{K}(t) = \sqrt{n}H_t(\widehat{\theta} - \theta^{\star})$ when evaluated in $t$. Assume for simplicity, and without loss of generality, that $t = v_j$ for some $1 \leq j \leq K$, then, using Proposition~\ref{prop::simple_clt1} and the expression in~\eqref{eq::beta_star_nice}, we can write    
\begin{equation}
\Uscr_{n}^K(t) 
\doteq \frac{1}{\sqrt{n}} \sum_{i=1}^n 
\begin{pmatrix} 
\sum_{j : v_j \leq t}G_j(\theta^{\star})^{-1}Z_i\int_{W_j}\dd M_i(s,\theta^{\star})\Delta_j \\
 \int_{0}^{\tau}\{\dot{\Psi}_{\theta^{\star}}^{K,10}h_{1:K}(s,X_i,Z_i,\theta^{\star})
 + \dot{\Psi}_{\theta^{\star}}^{K,11} h_{K+1}(s,X_i,Z_i,\theta^{\star})\}
\,\dd M_i(s,\theta^{\star})
\end{pmatrix},
\label{eq::MG_form_explicit}
\notag
 \end{equation}
as $n$ tends to infinity, where $a \doteq b$ means that $a = b + o_p(1)$. This representation of $\Uscr_n^K$ is used actively when studying the variance process $\langle \Uscr_n^K,\Uscr_n^K\rangle_t $ as $n$ and $K$ tends to infinity. Define the matrix function
\begin{equation}
\Sigma_{t}^K  
= 
\begin{pmatrix}
\Sigma_{t,00}^K & \Sigma_{t,01}^K\\
\Sigma_{t,10}^K & \Sigma_{11}^K
\end{pmatrix}
= \langle \Uscr_n^K,\Uscr_n^K\rangle_t,
\notag
\end{equation} 
where $\Sigma_{t,00}^K = n\langle \widehat{B} - B^{\star},\widehat{B} - B^{\star}\rangle_t$ and $\Sigma_{11}^K = n\langle\widehat{\gamma} - \gamma^{\star},\widehat{\gamma} - \gamma^{\star}\rangle_{\tau}$ are $q\times q$ and $p\times p$ matrices, respectively, while $\Sigma_{t,01}^K$ is $q \times p$ and $\Sigma_{t,10}^K$ is $p\times q$. The limit of $\Sigma_{t}^K$ as both $n$ and $K$ tend to infinity is denoted $\Sigma_t$, with $\Sigma_{t,00}$, $\Sigma_{t,01}$, $\Sigma_{t,10}$, and $\Sigma_{t,11}$ being its four blocks. Define the functions 
\begin{equation}
\begin{split}
\lambda_{\theta}(s) & = y(s;x,z)w(s,\theta),\quad \text{for $j = 1,\ldots,K$},\\
\mu_{\theta}(s) & = y(s;x,z)w(s,\theta)\{1 - w(s,\theta)\},\quad \text{for $j = 1,\ldots,K$},\\
\nu_{\theta}(s) & = y(s;x,z)w(s,\theta)\{1 - w(s,\theta)\}z^{\prime}\beta_j,\quad \text{for $j = 1,\ldots,K$},\\
\xi_{\theta}(s) & = y(s;x,z)w(s,\theta)\{1 - w(s,\theta)\}^2z^{\prime}\beta(s),
\end{split}
\label{eq::lambda_mu_nu_xi}
\end{equation}
and recall that all of these are bounded. We also need the function 
\begin{equation}
G(s,\theta) = \E\, ZZ^{\prime}\lambda_{\theta}(s) = \E\, ZZ^{\prime}y(s;X,Z)w(s,\theta).
\label{eq::Gfunc}
\end{equation}
Since $s \mapsto y(s;x,z)$ and $s \mapsto w(s,\theta)$ have bounded derivatives for all parameter values, and because the covariates are bounded with probability one, we see that
\begin{equation}
G_j(\theta^{\star}) = G(v_{j-1},\theta)\Delta_j + O(\Delta_j^2), \quad\text{for $j = 1,\ldots,K$}, 
\label{eq::Gjapprox}
\end{equation}
as $K \to \infty$. In Lemma~\ref{lemma:Psi11_approx} it is shown that for the lower right corner $\dot{\Psi}_{\theta^{\star}}^{K,11}$ of $\dot{\Psi}_{\theta^{\star}}^{-1}$ in~\eqref{eq::dotPsi}, 
\begin{equation}
\dot{\Psi}_{\theta^{\star}}^{K,11} = \dot{\Psi}^{11} + O(1/K),
\notag
\end{equation}
as the mesh size tends to zero, where $\dot{\Psi}^{11}$ is the inverse of the $p \times p$ matrix
\begin{equation}
(\dot{\Psi}^{11})^{-1} = \int_{0}^{\tau}\E\,XX^{\prime}\xi_{\varphi^{\circ}}(s)\,\dd s -  \int_{0}^{\tau}\E\,\{XZ^{\prime}\nu_{\varphi^{\circ}}(s)\}G(s,\varphi^{\circ})^{-1}\E\,\{ZX^{\prime}\mu_{\varphi^{\circ}}(s)\}\,\dd s. 
\notag
\end{equation}   
We have the following result.
\begin{lemma}\label{lemma::sigma_matrix_explicit} For fixed $t$, as $n \to \infty$ and $K \to \infty$, the matrix $\Sigma_{t}^K \to_p \Sigma_{t}$, with blocks 
\begin{equation}
\begin{split}
\Sigma_{t,00}^{K} \overset{p}\to \Sigma_{t,00} & = \int_0^t G(s,\varphi^{\circ})^{-1}\E\, Z\, y(s;X,Z) w(s,\varphi^{\circ}) Z^{\prime}\beta^{\circ}(s)Z^{\prime}G(s,\varphi^{\circ})^{-1}\,\dd s,\\
\Sigma_{11}^{K} \overset{p}\to \Sigma_{11} & = \dot{\Psi}^{11}\E\,\int_{0}^{\tau}  A(s,\varphi^{\circ})A(s,\varphi^{\circ})^{\prime}y(s,X,Z)w(s,\varphi^{\circ})Z^{\prime}\beta^{\circ}(s)\,\dd s\,(\dot{\Psi}^{11})^{\prime},\\   
\Sigma_{t,01}^{K} \overset{p}\to \Sigma_{t,01} & = \int_0^t G(s,\varphi^{\circ})^{-1} 
\E\, Z A(s,\varphi^{\circ})^{\prime} y(s;X,Z) Z^{\prime}\beta^{\circ}(s)\,\dd s\,(\dot{\Psi}^{11})^{\prime}, 
\end{split}
\notag
\end{equation}
and $\Sigma_{t,10}^{K} \to_p \Sigma_{t,10} = (\Sigma_{t,01})^{\prime}$, where $A(s,\varphi) \in \real^p$ is the function
\begin{equation}
\begin{split}
A(s,\varphi) = \E\,XZ^{\prime}\nu_{\varphi}(s)G(s,\varphi)^{-1} Z - X\{1 - w(s,\varphi)\}.
\end{split}
\notag
\end{equation}
\end{lemma}
\begin{proof} 
Since $\theta \mapsto G(s,\theta)$ and $\theta \mapsto w(s,\theta)$ are continuous and $H_t \theta^{\star} \to (B^{\circ}(t),\gamma^{\circ})^{\prime}$ as $K \to \infty$, we have that $G(s,\theta^{\star}) \to G(s,\varphi^{\circ})$ and $w(s,\theta^{\star}) \to w(s,\varphi^{\circ})$. This is used without further comment in the following.  Write $\langle \widehat{B} - B^{\star},\widehat{B} - B^{\star}\rangle_{W_j}=\langle \widehat{B} - B^{\star},\widehat{B} - B^{\star}\rangle_{v_j} - \langle \widehat{B} - B^{\star},\widehat{B} - B^{\star}\rangle_{v_{j-1}}$. Over an interval $W_j$, as $n \to \infty$,  
\begin{equation} 
\begin{split}
n\langle \widehat{B} - B^{\star},\widehat{B} - B^{\star}\rangle_{W_j}  
& = \frac{1}{n}\sum_{i=1}^n G_j(\theta^{\star})^{-1} Z_i\int_{W_j} Y_i(s) w_i(s,\theta^{\star}) Z_i^{\prime}\beta^{\star}_j\,\dd s \,Z_i^{\prime}G_j(\theta^{\star})^{-1}\Delta_j^2\\
& = G_j(\theta^{\star})^{-1} \E\,Z\,\int_{W_j} y(s;X,Z) w(s,\theta^{\star}) Z^{\prime}\beta^{\star}_j\,\dd s \,Z^{\prime}G_j(\theta^{\star})^{-1}\Delta_j^2 + o_p(1).
\end{split}
\notag
\end{equation}
From~\eqref{eq::Gjapprox} we get $G_j(\theta^{\star})  = G(v_{j-1},\theta^{\star})\Delta_j + O(\Delta_j^2)$ for $j = 1,\ldots,K$, as $K \to \infty$, where $G(s,\theta)$ is defined in~\eqref{eq::Gfunc}. A similar approximation for the nominator in $n\langle \widehat{B} - B^{\star},\widehat{B} - B^{\star}\rangle_{W_j}$ gives that 
\begin{equation}
n\langle \widehat{B} - B^{\star},\widehat{B} - B^{\star}\rangle_{W_j}  =
\{G(v_{j-1},\theta^{\star})\}^{-1} \E\,\lambda_{\theta^{\star}}(v_{j-1})Z^{\prime}\beta_j^{\star}  \{G(v_{j-1},\theta^{\star})\}^{-1} \Delta_j + O(\Delta_j^2),
\notag
\end{equation}
as $n \to \infty$ and $K \to \infty$. Then $\Sigma_{t,00}^{K} = n\langle \widehat{B} - B^{\star},\widehat{B} - B^{\star}\rangle_{t}$ is 
\begin{equation}
\begin{split}
\Sigma_{t,00}^{K} &
= \sum_{j : v_j \leq t}\{G(v_{j-1},\theta^{\star})\}^{-1} \E\,\lambda_{\theta^{\star}}(v_{j-1})Z^{\prime}\beta_j^{\star}  \{G(v_{j-1},\theta^{\star})\}^{-1} \Delta_j + O(1/K),
\end{split}
\notag
\end{equation}
which converges to $\Sigma_{t,00}$ as $n \to \infty$ and $K\to\infty$. As a consequence of Lemma~\ref{lemma:Psi11_approx} we have that
\begin{equation}
\dot{\Psi}_{\theta^{\star}}^{K,10}h(s,X_i,Z_i,\theta^{\star}) 
= - \dot{\Psi}^{11} \sum_{j=1}^K \,\E\,\{XZ^{\prime}\nu_{\theta^{\star}}(v_{j-1})\}G(v_{j-1},\theta^{\star})^{-1} Z_iI_{W_j}(s)
+ O(1/K).
\notag
\end{equation}  
Write $A_{i,j}(s,\theta) = \E\,\{XZ^{\prime}\nu_{\theta}(v_{j-1})\}G(v_{j-1},\theta^{\star})^{-1} Z_iI_{W_j}(s) - X_i\{1 - w_i(s,\theta)\}$, and $A_j(s,\theta) = \E\,\{XZ^{\prime}\nu_{\theta}(v_{j-1})\}G(v_{j-1},\theta^{\star})^{-1} ZI_{W_j}(s) - X\{1 - w(s,\theta)\}$. Then, also using Lemma~\ref{lemma:Psi11_approx},
\begin{equation}
\begin{split}
\Sigma_{11}^K & 
= \frac{1}{n} \sum_{i=1}^n \dot{\Psi}_{\theta^{\star}}^{K,11}\sum_{j=1}^K\int_{W_j} A_{i,j}(s,\theta^{\star})A_{i,j}(s,\theta^{\star}) Y_i(s)w_i(s,\theta^{\star})Z_i^{\prime}\beta^{\star}(s)\,\dd s\, (\dot{\Psi}_{\theta^{\star}}^{K,11})^{\prime}\\ 
& = \dot{\Psi}^{11}\sum_{j=1}^K\E\,\int_{W_j} A_j(s,\theta^{\star})A_j(s,\theta^{\star})^{\prime} \lambda_{\theta^{\star}}(s)Z^{\prime}\beta^{\star}(s)\,\dd s\, (\dot{\Psi}^{11})^{\prime} + O_p(1/K) + o_p(1),
\end{split}
\notag
\end{equation}
as $n,K \to \infty$. Since all the elements of the $A_j(s,X,Z,\theta^{\star})$ are bounded with probability one by Assumption~\ref{assumption3}, 
\begin{equation}
\begin{split}
& \sum_{j=1}^K\E\,\int_{W_j} A_j(s,\theta^{\star})A_j(s,\theta^{\star})^{\prime} \lambda_{\theta^{\star}}(s)Z^{\prime}\beta^{\star}(s)\,\dd s\\  
& \qquad \qquad = \sum_{j=1}^K\E\,A_j(v_{j-1},\theta^{\star})A_j(v_{j-1},\theta^{\star})^{\prime} \lambda_{\theta^{\star}}(v_{j-1})Z^{\prime}\beta_j^{\star}\Delta_j + O_p(1/K) ,
\end{split}
\notag
\end{equation}
whose limit is $\E\,\int_0^{\tau} A(s,\varphi^{\circ})A(s,\varphi^{\circ})^{\prime} y(s;X,Z)w(s,\varphi^{\circ})Z^{\prime}\beta^{\circ}(s)\,\dd s$ as $K \to \infty$. By similar arguments and using the same notation,  
\begin{equation}
\begin{split}
\Sigma_{t,01}^K & = \frac{1}{n}\sum_{i=1}^n\sum_{j : v_j \leq t}
G_j(\theta^{\star})^{-1} Z_i \int_{W_j} A_{i,j}(s,\theta^{\star})^{\prime}  
Y_i(s)w(s,\theta^{\star}) Z_i^{\prime}\beta_j^{\star}\,\dd s\,(\dot{\Psi}^{11})^{\prime}\Delta_j + O_p(1/K^2)\\
&
= \sum_{j : v_j \leq t}
G_j(\theta^{\star})^{-1}\,\E\, Z \int_{W_j} A_j(s,\theta^{\star})^{\prime} 
\lambda_{\theta^{\star}}(s) Z^{\prime}\beta_j^{\star}\,\dd s\,(\dot{\Psi}^{11})^{\prime}\Delta_j  + O_p(1/K^2) + o_p(1)\\
& 
= \sum_{j : v_j \leq t}
G(v_{j-1},\theta^{\star})^{-1}\,\E\, Z A_j(v_{j-1},\theta^{\star})^{\prime} 
\lambda_{\theta^{\star}}(v_{j-1}) Z^{\prime}\beta_j^{\star}\,(\dot{\Psi}^{11})^{\prime}\Delta_j  + O_p(1/K^2) + o_p(1),
\end{split}
\notag
\end{equation}
and the claim follows by letting $K \to \infty$.
\end{proof}
Recall from~\eqref{eq::Z_n_K} that $\Uscr_{n}^K = \sqrt{n}\{(\widehat{B} - B^{\star})^{\prime},(\widehat{\gamma} - \gamma^{\star})^{\prime}\}^{\prime}$. The first, triangular array, process convergence result takes place under the sequence $(P_{\star}^K)_{K \geq 1}$ of measures. 

\begin{theorem}\label{theorem::process_clt_1} As $n \to \infty$ and $K\to\infty$, 
\begin{equation}
\Uscr_{n}^K \Rightarrow \Uscr,
\notag
\end{equation}
under the sequence $(P_{\star}^{K})_{K}$, where $\Uscr = (\Uscr_1,\ldots,\Uscr_{q+p})$ is a Gaussian martingale with mean zero and variance process $\Sigma_{t}$, where $\Sigma_t$ is given in Lemma~\ref{lemma::sigma_matrix_explicit}. 
\end{theorem} 
\begin{proof} From the expression for $\Uscr_n^K$ below~\eqref{eq::Z_n_K} we see that 
\begin{equation}
\Uscr_n^K(t) = n^{-1/2}\sum_{i=1}^n\int_0^t \Xi^n(s,X_i,Z_i)\,\dd M_i(s),
\notag
\end{equation} 
in terms of functions $\Xi^n(s,X_i,Z_i)$ with values in $\real^{q+p}$. From Assumptions~\ref{assumption1}--\ref{assumption3}, there is a constant $b<\infty$, such that $\norm{\Xi^n(s,X_i,Z_i)}\leq b$. From Lemma~\ref{lemma::sigma_matrix_explicit} we have $\langle \Uscr_n^K,\Uscr_n^K\rangle \to_p \Sigma_t$ as $n,K\to \infty$, where each component of $\Sigma_t$ is absolutely continuous with respect to Lebesgue measure. The claim then follows from Theorem~\ref{theorem::counting_clt} in Appendix~\ref{app::lemma::uniqueness}.
\end{proof}

\subsection{Shifting back to $P_{\circ}$}\label{sec::shifting_back}
In this section we use measure change techniques closely associated with those developed and studied in~\citet{mykland2009inference}, to obtain large-sample results for the estimators in~\eqref{eq::betahat1} when $n$ and $K$ tend to infinity, under the distribution $P_{\circ}$ associated with the survival function in~\eqref{eq::popsurv1}, that is, under the parameter values $\gamma_0^{\circ},\ldots,\gamma_{p-1}^{\circ}$ and $\beta_0^{\circ},\ldots,\beta_{q-1}^{\circ}$. The measure change techniques we use are akin to process convergence analogues of Le Cam{'}s third lemma (see e.g.~\citet[p.~90]{vanderVaart1998}, or~\citet[Def.~V.1.1, p.~285]{jacod2003limit}). The log-likelihood of the model presented in Section~\ref{sec::2.2} is (see e.g.~\citet[p.~98]{andersen1993statistical})
\begin{equation}
\ell_n(\varphi^{\circ}) = \sum_{i=1}^n \int_{0}^{\tau}[\log \{w(s,\varphi^{\circ}) Z_i^{\prime}\beta^{\circ}(s) \} \,\dd N_i(s) - Y_i(s)w(s,\varphi^{\circ}) Z_i^{\prime}\beta^{\circ}(s)\,\dd s ].
\notag
\end{equation} 
The log of the Radon--Nikodym derivative $\dd P_{\circ,n}/\dd P_{\star,n}^K$ is  
\begin{equation}
\begin{split}
\log \frac{\dd P_{\circ,n}}{\dd P_{\star,n}^K } 
& = \sum_{i=1}^n \int_{0}^{\tau}\big[\log \frac{w(s,\varphi^{\circ}) Z_i^{\prime}\beta^{\circ}(s)}{w(s,\theta^{\star}) Z_i^{\prime}\beta^{\star}(s)} \,\dd N_i(s)\\ 
& \qquad \qquad \qquad \qquad - Y_i(s)\{w(s,\varphi^{\circ})Z_i^{\prime}\beta^{\circ}(s) - w(s,\theta^{\star}) Z_i^{\prime}\beta^{\star}(s)\}\,\dd s \big].
\end{split}
\notag
\end{equation} 
Define 
\begin{equation}
\phi_K(s) = \sqrt{n}\{\beta^{\circ}(s) - \beta^{\star}(s)\}, 
\label{eq::phi_Kn}
\end{equation}
so that $\beta^{\circ}(s) = \beta^{\star}(s) + \phi_K(s)/\sqrt{n}$. Recall that $\beta^{\star}(s) = \sum_{j=1}^n \beta_j^{\star}I_{W_j}(s)$ with $\beta_j^{\star} = \beta^{\circ}(v_{j-1})$ for $j = 1,\ldots,K$. Under Assumption~\ref{assumption1}, the difference $\beta^{\circ}(s) - \beta^{\star}(s)$ tends to zero uniformly in $s$ as $K \to \infty$.

\begin{lemma}\label{lemma::lr1} The log-likelihood ratio $\log ( \dd P_{\circ,n}/\dd P_{\star,n}^K)$, evaluated in $t$, is  
\begin{equation}
\begin{split}
\log \frac{\dd P_{\circ,n}}{\dd P_{\star,n}^K }\big|_t 
& = \sum_{i=1}^n \big\{ \zeta_{i}^{K}(t) - \frac{1}{2}\langle \zeta_{i}^{K},\zeta_{i}^{K}\rangle_t \big\} + o_p(1),
\end{split}
\notag
\end{equation}
as $K \to \infty$, where $\zeta_{1}^{K},\ldots,\zeta_{n}^{K}$ are the $P_{\star}^K$ martingales given by 
\begin{equation}
\zeta_{i}^{K}(t) = \frac{1}{\sqrt{n}}\int_0^{t} a_{K,i}(s) \,\dd M_i(s,\theta^{\star}),
\label{eq::key_approx_MG} 
\end{equation}
with $a_{K,i}(s)  = \{Z_i^{\prime}\phi_K(s)\}/\{Z_i^{\prime}\beta^{\star}(s)\} - \{1 - w_i(s,\theta^{\star})\}\big(Z_i^{\prime}\int_0^s \phi_K(u)\,\dd u \big)$ for $i = 1,\ldots,n$.
\end{lemma}
\begin{proof} The log-likelihood ratio evaluated in $t$ is 
\begin{equation}
\begin{split}
\log \frac{\dd P_{\circ,n}}{\dd P_{\star,n}^K } \big|_{t}
& = \sum_{i=1}^n \int_{0}^{t}\bigg[\log \frac{w(s,\theta^{\circ}) Z_i^{\prime}\{\beta^{\star}(s) + \phi_K(s)/\sqrt{n}\}}{w(s,\varphi^{\star}) Z_i^{\prime}\beta^{\star}(s)} \,\dd N_i(s)\\ 
& \qquad\qquad- Y_i(s)[w(s,\varphi^{\circ})Z_i^{\prime}\phi_K(s)/\sqrt{n}
+ \{w(s,\varphi^{\circ}) - w(s,\theta^{\star}) \}Z_i^{\prime}\beta^{\star}(s)]\,\dd s \bigg]\\
& = \sum_{i=1}^n\bigg[ \int_{0}^{t}\big\{\log \frac{Z_i^{\prime}\{\beta^{\star}(s) + \phi_K(s)/\sqrt{n}\}}{Z_i^{\prime}\beta^{\star}(s)}\,\dd N_i(s) - Y_i(s)w_i(s,\varphi^{\circ})Z_i^{\prime} \phi_K(s)\,\dd s\big\}\\
&  \qquad\qquad + \int_{0}^{t}\big\{ \log \frac{w_i(s,\varphi^{\circ})}{w_i(s,\theta^{\star})}\,\dd N_i(s) 
- Y_i(s) \{w(s,\varphi^{\circ}) - w(s,\theta^{\star})\}Z_i^{\prime}\beta^{\star}(s)\,\dd s\bigg].
\end{split}
\notag
\end{equation} 
Note that $\norm{\phi_K(s)} 
= \norm{\beta^{\circ}(s) - \beta^{\circ}(v_{j-1})} \leq \sup_{s\in W_j}\norm{\dot{\beta}^{\circ}(s)}|s - v_{j-1}|$, for $s \in W_j$. For $s$ fixed, and as $K \to \infty$, 
\begin{equation}
\begin{split}
w(s,\varphi^{\circ}) - w(s,\theta^{\star})& 
= -w(s,\theta^{\star})\{1 - w(s,\theta^{\star})\}Z_i^{\prime}\int_0^s \phi_K(u)/\sqrt{n}\,\dd u \\
& \qquad 
+ \frac{1}{2}w(s,\theta^{\star})\{1 - w(s,\theta^{\star})\}\\
& \qquad \qquad \times\{1 - 2w(s,\theta^{\star})\}\big(Z_i^{\prime}\int_0^s \phi_K(u)/\sqrt{n}\,\dd u\big)^2 + O_p(1/K^4),
\end{split}
\notag
\end{equation}
and, writing $\log (a/b) = \log \{ 1 + (a - b)/b\}$, 
\begin{equation}
\begin{split}
\log \frac{w(s,\varphi^{\circ})}{w(s,\theta^{\star})} 
& = \frac{w(s,\varphi^{\circ}) - w(s,\theta^{\star})}{w(s,\theta^{\star})}
- \frac{1}{2}\big(\frac{w(s,\varphi^{\circ}) - w(s,\theta^{\star})}{w(s,\theta^{\star})}\big)^2 + O_p(1/K^4),
\end{split}
\notag
\end{equation}
while
\begin{equation}
\begin{split}
\log \frac{Z_i^{\prime}\{\beta^{\star}(s) + \phi_K(s)/\sqrt{n}\} }{Z_i^{\prime}\beta^{\star}(s)} & = 
\frac{Z_i^{\prime}\phi_K(s)}{Z_i^{\prime}\beta^{\star}(s)} - \frac{1}{2}\frac{\{Z_i^{\prime}\phi_K(s)/\sqrt{n}\}^2}{\{Z_i^{\prime}\beta^{\star}(s)\}^2} + O_p(1/K^3),
\end{split}
\notag
\end{equation}
where the probability is with respect to the distribution of the covariates. Consider only the $i${'}th term in the sum in $\log (\dd P_{\circ,n}/\dd P_{\star,n}^K)|_{t}$. Then the two first terms on the right are, using~\eqref{eq::phi_Kn}, 
\begin{equation}
\begin{split}
& \int_{0}^{t}\big\{\log \frac{Z_i^{\prime}\{\beta^{\star}(s) + \phi_K(s)/\sqrt{n}\}}{Z_i^{\prime}\beta^{\star}(s)}\,\dd N_i(s) - Y_i(s)w_i(s,\varphi^{\circ})Z_i^{\prime} \phi_K(s)/\sqrt{n}\,\dd s\big\}\\
& \qquad\qquad = \frac{1}{\sqrt{n}}\int_0^{t}\frac{Z_i^{\prime}\phi_K(s)}{Z_i^{\prime}\beta^{\star}(s)}\,\dd M_i(s,\theta^{\star})
- \frac{1}{2n}\int_0^{t}\big\{\frac{Z_i^{\prime}\phi_K(s)}{Z_i^{\prime}\beta^{\star}(s)}\big\}^2\,\dd\langle M_i(\cdot,\theta^{\star}),M_i(\cdot,\theta^{\star})\rangle_s\\ 
& \qquad \qquad \qquad \qquad \qquad \qquad - \frac{1}{2n}\int_0^{t}\big\{\frac{Z_i^{\prime}\phi_K(s)}{Z_i^{\prime}\beta^{\circ}(s)}\big\}^2\,\dd M_i(s,\theta^{\star}) + O_p(1/K^3).
\end{split}
\notag
\end{equation}
For the last two terms on the right hand side in $\log (\dd P_{\circ,n}/\dd P_{\star,n}^K)|_{t}$ (the $i${'}th term only),
\begin{equation}
\begin{split}
& \int_{0}^{t}\big\{ \log \frac{w_i(s,\varphi^{\circ})}{w_i(s,\theta^{\star})}\,\dd N_i(s) 
- Y_i(s) \{w_i(s,\varphi^{\circ}) - w_i(s,\theta^{\star}) \}Z_i^{\prime}\beta^{\star}(s)\,\dd s\\
& \quad = \int_0^{t} \frac{w_i(s,\varphi^{\circ}) - w_i(s,\theta^{\star})}{w_i(s,\theta^{\star})}\, \dd M_i(s,\theta^{\star}) 
 - \frac{1}{2}\int_0^{t} \big(\frac{w_i(s,\theta^{\star}) - w_i(s,\varphi^{\circ})}{w_i(s,\varphi^{\circ})}\big)^2\,\dd N_i(s) + O_p(1/K^4)\\
& \quad= -\frac{1}{\sqrt{n}}\int_0^{t}\{1 - w_i(s,\theta^{\star})\} \big( Z_i^{\prime}\int_0^s \phi_K(u)\,\dd u \big)\, \dd M_i(s,\theta^{\star})\\
& \qquad\qquad -\frac{1}{2n}\int_0^{t} \big[\{1 - w_i(s,\theta^{\star})\} \big( Z_i^{\prime}\int_0^s \phi_K(u)\,\dd u \big)\big]^2\\ 
& \qquad\qquad\qquad\qquad\times\big\{\dd M_i(s,\theta^{\star}) + Y_i(s)w_i(s,\theta^{\star}) Z_i^{\prime}\beta^{\star}(s)\,\dd s\big\} 
+ O_p(1/K^4).
\end{split}
\notag
\end{equation}
Collecting the terms from the last two displays, and using that both 
\begin{equation}
\frac{1}{n}\sum_{i=1}^n\int_0^{t}\frac{\{Z_i^{\prime}\phi_K(s)\}^2}{Z_i^{\prime}\beta^{\star}(s)}\,\dd M_i(s,\theta^{\star}),\;\text{and}\; \frac{1}{n}\sum_{i=1}^n\int_0^{t} [\{1 - w_i(s,\theta^{\star})\} \big( Z_i^{\prime}\int_0^s \phi_K(u)\,\dd u \big)]^2\,\dd M_i(s,\theta^{\star}), 
\notag
\end{equation}
are $O_p(1/\sqrt{n})$, uniformly in $t$ by Lenglart{'}s inequality \citep[Lemma~I.3.30, p.~35]{jacod2003limit}, we obtain the claim.
\end{proof}

If the $\Delta_1,\ldots,\Delta_K$ are proportional to $1/\sqrt{n}$, then there is a function $\phi(s)$ such that 
\begin{equation}
\int_0^t \phi_K(s)\,\dd s \to \int_0^t \phi(s)\, \dd s,
\label{eq::phi_convergence}
\end{equation}
for all $t$, as $n,K \to \infty$, where $\phi(s) \in \real^q$ is a function not identically equal to zero. This follows from Assumption~\ref{assumption1}, and is seen by looking at $\phi_K(s)$ component for component. Assume for simplicity that $\phi_K(s)$ is one-dimensional; that $\Delta_j = 1/\sqrt{n}$ for all $j = 1,\ldots,K$, i.e.~$\tau = 1$, and that $t = v_{\ell}$, then  
\begin{equation}
\begin{split}
\int_0^t \phi_K(s)\,\dd s & = \sqrt{n}\int_0^t \{\beta^{\circ}(s) - \beta^{\star}(s)\}\,\dd s
 = \sum_{j : v_{j}\leq t}\frac{1}{\Delta_j}\int_{v_{j-1}}^{v_{j}} \{\beta^{\circ}(s) - \beta^{\circ}(v_{j-1})\}\,\dd s  \\
& = \sum_{j : v_{j}\leq v_{\ell}}\frac{1}{\Delta_j}\int_{v_{j-1}}^{v_{j}}\{ \dot{\beta}^{\circ}(v_{j-1})(s - v_{j-1}) + \frac{1}{2}\ddot{\beta}^{\circ}(s_j)(s - v_{j-1})^2 \} \,\dd s\\
& = \frac{1}{2}\sum_{j : v_{j}\leq v_{\ell}}\dot{\beta}^{\circ}(v_{j-1})\Delta_j + O(\max_{j \leq K} \Delta_j),
\end{split}
\label{eq::cumulative_approx}
\end{equation}
where $s_j$ is some point between $s$ and $v_{j-1}$, hence dependent on the partition; and the right hand side is a Riemann sum tending to $\int_0^{v_{\ell}}\beta^{\circ}(s)/2\,\dd s$ as $K \to \infty$. In this simple case, this means that $\phi(s) = \beta^{\circ}(s)/2$. The same argument goes through for $\int_0^t \phi_K(s)^2\,\dd s$. 

The generic version of the $a_{K,i}(s)$ defined in~\eqref{eq::key_approx_MG} is the random variable 
\begin{equation}
a_K(s) = \{Z^{\prime}\phi_K(s)\}/\{Z^{\prime}\beta^{\star}(s)\} - \{1- w(s,\theta^{\star})\}\big(Z^{\prime}\int_0^s \phi_K(u)\,\dd u \big), 
\notag
\end{equation}
and $a(s)$ is the same thing but evaluated in $\varphi^{\circ}$ and $\phi(s)$ instead of $\theta^{\star}$ and $\phi_K(s)$. 

\begin{corollary} Assume that $K$ is of order $\sqrt{n}$. Then, 
\begin{equation}
\log \frac{\dd P_{\circ,n}}{\dd P_{\star,n}^K} \Rightarrow \log L,
\notag
\end{equation}
under $P_{\star,n}^{K}$ as $n \to \infty$ and $K \to \infty$, where $L$ is the process $L_t = \exp(-\int_0^t\sigma(s)^2\,\dd s/2 + \int_0^t\sigma(s) \,\dd W_s)$, with $W$ is a standard Wiener process, and  
\begin{equation}
\sigma(s) = \E\, \big[ \frac{Z^{\prime}\phi(s)}{Z^{\prime}\beta^{\circ}(s)} - \{1 - w(s,\varphi^{\circ})\}\big(Z^{\prime}\int_0^s \phi(u)\,\dd u \big)  \big]\{y(s;X,Z) w(s,\varphi^{\circ})Z^{\prime}\beta^{\circ}(s)\}^{1/2}.
\notag
\end{equation}
\end{corollary} 
\begin{proof} 
The martingale $\sum_{i=1}^n\zeta_i^{K}(t) = n^{-1/2}\sum_{i=1}^n\int_0^t a_{K,i}(s)\,\dd M_i(s,\theta^{\star})$ is the sum of the martingales $\zeta_1^{K},\ldots,\zeta_n^{K}$ defined in~\eqref{eq::key_approx_MG}. Since the counting processes are independent, the quadratic variation is $n^{-1}\sum_{i=1}^n\langle\zeta_i^{K},\zeta_i^{K} \rangle_t = n^{-1}\sum_{i=1}^n \int_0^t a_{K,i}(s)^2Y_i(s)w(s,\theta^{\star})Z_i^{\prime}\beta^{\star}(s)\,\dd s$ which converges in probability to $\E\, \int_0^t a_{K}(s)^2y(s;X,Z)w(s,\theta^{\star})Z^{\prime}\beta^{\star}(s)\,\dd s$ for all $t$ as $n \to \infty$, provided $\Delta_j \propto 1/\sqrt{n}$ for $j = 1,\ldots,K$; and by Assumptions~\ref{assumption1}--\ref{assumption2}, 
\begin{equation}
\E\, \int_0^t a_{K}^2(s)y(s;X,Z)w(s,\theta^{\star})Z^{\prime}\beta^{\star}(s)\,\dd s \to \int_0^t\sigma(s)^2\,\dd s,
\notag
\end{equation} 
for all $t$ as $K\to \infty$ by dominated convergence, using~\eqref{eq::cumulative_approx} (see the proof of Lemma~\ref{app::lemma_lr2} for more detail). By Assumptions~\ref{assumption1}--\ref{assumption3}, the functions $a_{K,i}(s)$ are all bounded above, so claim follows from Theorem~\ref{theorem::counting_clt} in Appendix~\ref{app::lemma::uniqueness}.
\end{proof}
Since the limiting process has $\E_{\varphi^{\circ}}\, L_t = 1$ for all $t$,
this corollary entails that $P_{\circ,n}$ is contiguous with respect to $P_{\star,n}^K$ (see e.g.~\citet[Lemma~6.4, p.~88]{vanderVaart1998}, \citet[Corollary~V.1.12, p.~289]{jacod2003limit}, or~\citet[Remark~2, p.~1411]{mykland2009inference}). 

Define the sequence of processes given by,
\begin{equation}
c_t^{K,n} = \big\langle \Uscr_n^K,\log \frac{\dd P_{\circ,n}}{\dd P_{\star,n}^K}  \big\rangle_t, \quad \text{for $0 \leq t \leq \tau$}.
\label{eq::c_t_Kn}
\end{equation}
In the next lemma we find the probability limit of $c_t^{K,n}$, and then change back to convergence of $\Uscr_n^K$ under the distribution $P_{\circ}$ associated with the continuous regression functions $\beta_0^{\circ}(s),\ldots,\beta_{q-1}^{\circ}(s)$. 

\begin{lemma}\label{eq::loglik_covar} Assume that $K$ is of the order $\sqrt{n}$, then $c_t^{K,n} \to_p c_t = (c_{t,1}^{\prime},c_{t,2}^{\prime})^{\prime}$ as $n$ and $K$ tend to infinity, where
\begin{equation}
\begin{split}
c_{t,1} & = \int_0^t G(s,\varphi^{\circ})^{-1}\,\E\, Z a(s) y(s;X,Z)w(s,\varphi^{\circ}) Z^{\prime}\beta^{\circ}(s)\,\dd s,\\
c_{t,2} & = (\dot{\Psi}^{11})^{\prime}\int_0^t\E\, A(s,\varphi^{\circ})^{\prime}a(s)y(s;X,Z)w(s,\varphi^{\circ}) Z^{\prime}\beta^{\circ}(s)\,\dd s,
\end{split}
\notag
\end{equation}
which are of dimensions $q\times 1$ and $p\times 1$, respectively.
\end{lemma}
\begin{proof} With $a_{K,i}(s)$ defined as in~\eqref{eq::key_approx_MG}, and using the expression for $\Uscr_n^K$ given on page~\pageref{eq::MG_form_explicit}, then a straight forward calculation yields that for $c_{t}^{K,n} = \{(c_{t,1}^{K,n})^{\prime},(c_{t,2}^{K,n})^{\prime}\}^{\prime}$, 
\begin{equation}
\begin{split}
c_{t,1}^{K,n} & = \frac{1}{n}\sum_{i=1}^n\sum_{j:v_j \leq t}G_j(\theta^{\star})^{-1}\Delta_jZ_i\int_{W_j}a_{K,i}(s)Y_i(s) w_i(s,\theta^{\star})Z_i^{\prime}\beta^{\star}(s)\,\dd s \\
& = \E\,\sum_{j:v_j \leq t}G_j(\theta^{\star})^{-1}\Delta_j Z\int_{W_j}a_{K}(s)\lambda_{\theta^{\star}}(s)Z^{\prime}\beta^{\star}(s)\,\dd s  + o_p(1)\\
& = \,\E\,\sum_{j:v_j \leq t}G(v_{j-1},\theta^{\star})^{-1} Za_{K}(v_{j-1})
\lambda_{\theta^{\star}}(v_{j-1})Z^{\prime}\beta^{\star}(v_{j-1})\,\Delta_j + O(1/K^2) + o_p(1),
\end{split}
\notag
\end{equation} 
as $n,K \to \infty$. The right hand side is a Riemann sum, and by the same argument used in the proof of Lemma~\ref{app::lemma_lr2}, $c_{t,1}^{K,n}$ tends to $c_{t,1}$ as $K \to \infty$, where we use that the functions $a_K(s)$, $y(s;x,z)$, and $w(s,\theta)$ are all continuous in the parameters. A similar argument gives $c_{t,2}$.  
\end{proof}
 
\begin{theorem}\label{theorem::process_clt_II} Provided $\Delta_1,\ldots,\Delta_k$ are of order $1/\sqrt{n}$ or smaller, 
\begin{equation}
\Uscr_{n}^{K} \Rightarrow c + \Uscr,
\notag
\end{equation}
under $P_{\circ}$ as $n\to \infty$ and $K \to \infty$. The function $c_t$ is the probability limit of $c_{t}^{K,n}$ as defined in~\eqref{eq::c_t_Kn} under $P_{\star}^K$ as $n,K \to \infty$; and where $\Uscr = (\Uscr_1,\ldots,\Uscr_{q+p})$ is a mean zero Gaussian martingale with covariance function $\cov(\Uscr_t,\Uscr_s) = \Sigma_{t \wedge s}$, with $\Sigma_t$ given in Lemma~\ref{lemma::sigma_matrix_explicit}. In particular, the function $c_t \equiv 0$ if $\sqrt{n}\max_{j \leq K}\Delta_j \to 0$.
\end{theorem}
\begin{proof} Since the processes $\Uscr_n^K$ and $\log ( \dd P_{\circ,n}/\dd P_{\star,n}^K)$ have the same driving martingales, joint weak convergence of $\{\Uscr_n^K,\log ( \dd P_{\circ,n}/\dd P_{\star,n}^K)\}$ under $P_{\star}^K$ can be deduced from Theorem~\ref{theorem::process_clt_1}, Lemma~\ref{lemma::lr1}, Lemma~\ref{eq::loglik_covar}, using Theorem~\ref{theorem::counting_clt}. The claim, namely joint convergence under $P_{\circ}$, then follows from a general version of Le Cam{'}s third lemma (see \citet[Theorem 6.6, p.~90]{vanderVaart1998} or \citet[Lemma~V.1.13, p.~289]{jacod2003limit}).  
\end{proof}

\begin{remark} The condition $\sqrt{n}\max_{j \leq K}\Delta_j \to 0$ will be met again in Appendix~\ref{app::aalens_linear_pfs}, where we derive the limiting distribution of the Aalen estimator $\widetilde{B}$ of~\eqref{eq::standard.est}, both {`}directly{'} under the $P_{\circ}$ distribution, and by contiguity techniques {\`a} la those of the preceding theorem.
\end{remark}

\section{Simulations}\label{sec::simulations} 
Figure~\ref{fig::fig1} contains plots of the estimated parameters $\widehat{\gamma}_0$ and $\widehat{\gamma}_1$, as well as the estimated cumulative regression function $B_l(t) = \int_0^t \beta_l(s)\,\dd s$ for $l = 0,1,2$, with $\beta_0(t) = 1/2 + (1/2)t\{1 - (1/2)t\}$, $\beta_1(t) = t/(1/2 - 8t^2)$, and $\beta_2(t) = - \beta_1(t)/2$, for $0 \leq t \leq 2$. The data were simulated from a cure model of the form~\eqref{eq::popsurv1}, with $n = 5000$; $\gamma_0 = 0.54$, $\gamma_1 = 2.60$, and the $X_i$ independent standard normals; and the hazard of the $i${'}th individual being $\beta_0(t) + \beta_1(t)Z_{i,1} + \beta_2(t)Z_{i,2}$, with $Z_{i,1}$ and $Z_{i,2}$ independent mean zero normals with variance $1/4$. The censoring variables were independent draws from a uniform distribution on $(0,2)$. We ran $100$ simulations, and sampled new covariates for each simulation. The parameters were estimated using the algorithm in~\eqref{eq::the_alg1} of Appendix~\ref{app::algorithm1}.    

\begin{figure}
\centering
\includegraphics[scale=0.45,angle = 270]{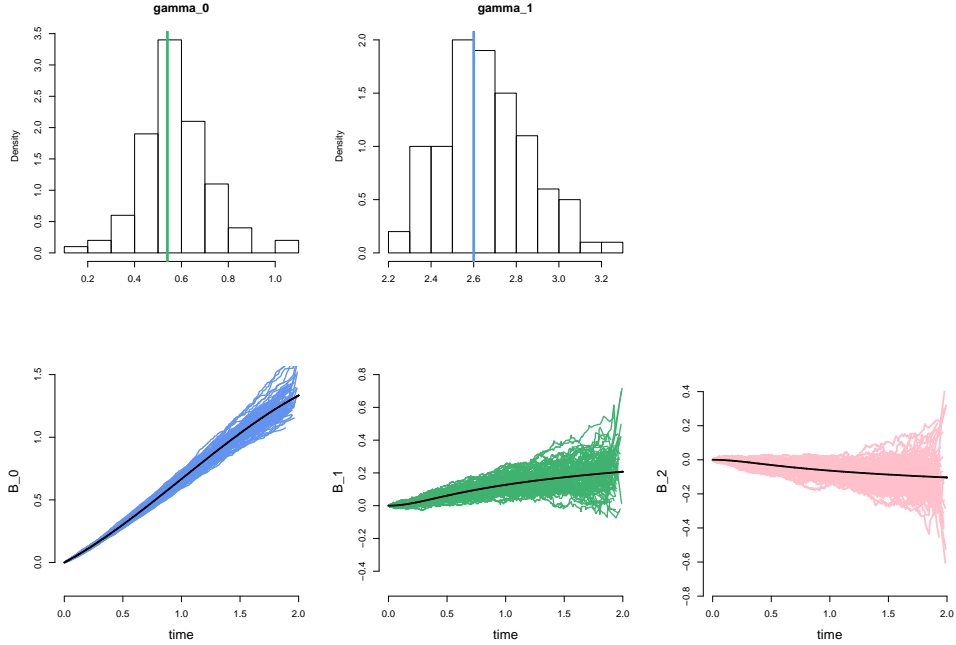} 
\caption{Upper panel: Estimates of $\gamma_0$ and $\gamma_1$, with the vertical lines indicating the true values. Lower panel: The nonparametric estimators $\widehat{B}_0(t), \widehat{B}_1(t)$, and $\widehat{B}_2(t)$, see eq.~\ref{eq::B_hat_nonpara}, with black lines indicating the true cumulative regression functions. The simulations on which these plots are based, are described in Section~\ref{sec::simulations}.} 
\label{fig::fig1}
\end{figure}


\section{Concluding remarks}
The techniques used in this technical report to arrive at process convergence results for a semiparametric survival model are, to my knowledge, new in survival analysis. As mentioned, similar techniques have been studied in the high-frequency literature by~\citet{mykland2009inference}, where stochastic volatility It{\^o}-process models are discretised, easing the derivation of estimators, as well as the study of their asymptotic properties. The semi-canonical strategy for studying the asymptotics of estimators in semiparametric models in survival analysis where the martingale theory is not immediately available, has become the methods initiated by~\citet{gill1989non} and by~\citet{murphy1995asymptotic}. See Appendix~\ref{subapp::spurver} for an example. I believe that the model studied in this technical report is easier handled with the methods employed here, rather than with the nonparametric likelihood theory used by~\citet{murphy1995asymptotic} for the gamma frailty model (and by the papers that tweak her proof to apply to other models). For future research, it would be interesting to investigate other problems in survival analysis that might be easier handled by employing parametric approximations and contiguity, rather than other available techniques. The same remark extends, of course, to semiparametric models beyond the survival analysis world. 

There are also Bayesian connections worth exploring here. \citet{hermansen2015bernshteuin} study parametric priors (as opposed to prior {\it processes}) for unknown functions $f$ that set $f$ constant over windows, as in this paper, and then let the number of windows increase with the sample size. Nonparametric cure models are complicated to handle in a Bayesian manner. The theory of this paper could possibly be merged with the theory of \citet{hermansen2015bernshteuin} to enable straighforward nonparametric Bayesian inference, and large sample results.   





\appendix 
\section{Deriving the estimation equations}\label{app::some_results} 
For any possibly right-censored event time $T = T^*\wedge C$ observed on $[0,\tau]$, with $C$ and censoring variable and $\delta = I\{T^*\geq C\}$, one can write $\delta f(T) = \int_0^{\tau} f(s)\,\dd N(s)$ and $f(T)  = \int_0^{\tau} Y(s)\,\dd f(s)$, for any $f$, and for any $f$ with $f(0) = 0$, respectively. Recall also the definition of the $r_{j}(t)$ functions $r_{j}(t) = \int_{W_j} I\{t \geq s\} \,\dd s = \int_0^{t}I_{W_j}(s)\,\dd s$, so that $r_{j}(T_i) = \int_{W_j} Y_i(s) \,\dd s$. The derivative of $r_{j}(T_i)$ is $I_{W_j}(T_i)$. With the locally constant hazard $\beta(t) = \sum_{j=1}^K \beta_j I_{W_j}(t)$ and $\pi(v) = \exp(v)/\{1+\exp(v)\}$, we also recall that 
\begin{equation}
w_i(t,\theta) = \pi\big( X_i^{\prime}\gamma - Z_i^{\prime}\int_0^t \sum_{j=1}^K \beta_jI_{W_j}(s)\,\dd s\big)
= \pi\big( X_i^{\prime}\gamma - Z_i^{\prime}\int_0^t \beta(s)\,\dd s\big), 
\notag
\end{equation}
hence, except at points of discontinuity, 
\begin{equation}
\frac{\dd}{\dd t}w_i(t,\theta) = - w_i(t,\theta)(1 - w_i(t,\theta)) Z_i^{\prime}\sum_{j=1}^K \beta_jI_{W_j}(t).
\notag
\end{equation}
These results were used to get from~\eqref{eq::system1} to~\eqref{eq::system2}, as follows
\begin{equation}
\begin{split}
&\delta_iI_{W_j}(T_i) - \int_{W_j} Y_i(s)\,\dd s(\delta_i + (1-\delta_i)w_i(T_i)) Z_i^{\prime}\beta_j \\
& \qquad\qquad= \int_0^{\tau}I_{W_j}\,\dd N_i(s) - r_j(T_i)\delta_i (1 - w_i(T_i))Z_i^{\prime}\beta_j - r_j(T_i)w_i(T_i)Z_i^{\prime}\beta_j\\
& \qquad\qquad= \int_0^{\tau}I_{W_j}\,\dd N_i(s) - \int_0^{\tau}r_j(s)(1 - w_i(s))Z_i^{\prime}\beta_j\,\dd N_i(s) - r_j(T_i)w_i(T_i)Z_i^{\prime}\beta_j.
\end{split}
\label{eq::derive.MGlik}
\end{equation}
For $r_j(T_i)w_i(T_i)Z_i^{\prime}\beta_j$ we use that $f(t) = \int_0^{\tau}Y_i(s)\,\dd f(s)$, then
\begin{equation}
\begin{split}
\dd\{r_j(t)w_i(t)\} & = \{I_{W_j}(t)w_i(t)\}\,\dd t - \{r_j(t)w_i(t)(1 - w_i(t)) Z_i^{\prime}\sum_{j=1}^K \beta_jI_{W_j}(t)\}\,\dd t\\
& = \{I_{W_j}(t)w_i(t)\}\,\dd t - \{r_j(t)w_i(t)(1 - w_i(t)) Z_i^{\prime}\beta(t)\}\,\dd t .
\end{split}
\notag
\end{equation}
Inserting this in~\eqref{eq::derive.MGlik} gives  
\begin{equation}
\begin{split}
\delta_iI_{W_j}(T_i) - & \int_{W_j} Y_i(s)\,\dd s(\delta_i + (1-\delta_i)w_i(T_i)) Z_i^{\prime}\beta_j\\
& =\int_0^{\tau}I_{W_j}\,\dd N_i(s) - \int_0^{\tau}r_j(s)(1 - w_i(s))Z_i^{\prime}\beta_j\,\dd N_i(s) - r_j(T_i)w_i(T_i)Z_i^{\prime}\beta_j\\
& = \int_0^{\tau}I_{W_j}\,\dd N_i(s) - \int_0^{\tau}r_j(s)(1 - w_i(s))Z_i^{\prime}\beta_j\,\dd N_i(s)\\ 
& \qquad - \int_0^{\tau} Y_i(s)I_{W_j}(s)w_i(s)Z_i^{\prime}\beta_j\,\dd s
+ \int_0^{\tau} Y_i(s) r_j(s)w_i(s)(1 - w_i(s)) Z_i^{\prime}\beta_j Z_i^{\prime}\beta(s)\,\dd s\\
& = \int_{0}^{\tau}I_{W_j}\{\dd N_i(s) - Y_i(s)w_i(s)Z_i^{\prime}\beta(s)\,\dd s \}\\
& \qquad\qquad\qquad - \int_0^{\tau}r_j(s)(1 - w_i(s))Z_i^{\prime}\beta_j \{\dd N_i(s) - Y_i(s)w_i(s)Z_i^{\prime}\beta(s)\,\dd s \}\\
& = \int_{0}^{\tau}\big\{I_{W_j} - r_j(s)(1 - w_i(s))Z_i^{\prime}\beta_j\big\}\,\dd M(s,\theta),
\end{split}
\notag
\end{equation}
because $\int_0^{\tau}I_{W_j}(t) Z_i^{\prime}\beta(t)\,\dd t =  \int_0^{\tau}I_{W_j}(t) Z_i^{\prime}\beta_j\,\dd t$.

\section{An algorithm}\label{app::algorithm1}
%
Let $\ell_n(\gamma,U)$ be the log-likelihood function that we would have used to estimate $\gamma$ had the Bernoulli vector $U = (U_1,\ldots,U_n)$ in~\eqref{eq::curehazard} been observable. It is given by
\begin{equation}
\ell_n(\gamma,u) = \sum_{i=1}^n \{ u_i \log \pi(X_i^{\prime}\gamma)
+ ( 1- u_i)\log ( 1- \pi(X_i^{\prime}\gamma)) \}.
\notag
\end{equation}
Let $D = (D_1,\ldots,D_n)$, with $D_i = (t_i,\delta_i,x_i,z_i)$ for $i = 1,\ldots,n$ be the observed data. Since $\ell_n(\gamma,u)$ is linear in $u$, we get that $\E_{\varphi}\, \{\ell_n(\gamma,u)\mid D\} = \ell_n\{\gamma,\E_{\varphi}\,(u\mid D)\}$, where 
\begin{equation}
\E_{\varphi}\,(u_i \mid D_i) = \delta_i + (1 - \delta_i)w_i(T_i,\varphi).
\notag
\end{equation} 
The algorithm we use to estimate $\gamma^{\circ}$ and $B^{\circ}$ is as follows: First, pick some initial values $\varphi^{(0)}= (\gamma^{(0)},\beta_1^{(0)},\ldots,\beta_K^{(0)})$. Second, for some small $\varepsilon > 0$, iterate  
\begin{equation}
\begin{split}
\gamma^{(r+1)} & = {\rm arg}\max_{\gamma} \E\,\{ \ell_n(\gamma,u)  \mid D,\varphi^{(r)}\},\\
\dd B^{(r+1)}(t) &= G_{n}(t,\varphi^{(r)})^{-1}n^{-1}\sum_{i=1}^n Z_i \dd N_i(t),\\ 
\end{split}
\label{eq::the_alg1}
\end{equation}    
for $r = 1,2,3,\ldots$, until $\norm{\varphi^{(r+1)} - \varphi^{(r)}} < \varepsilon$, with $G_{n}(t,\varphi)$ as defined in~\eqref{eq::B_hat_nonpara}. The {`}locally constant{'} version of this algorithm, that is, the one finding the solution to~\eqref{eq::the_est_eq}, replaces~\eqref{eq::the_alg1} with 
\begin{equation}
\begin{split}
\gamma^{(r+1)} & = {\rm arg}\max_{\gamma} \E\,\{ \ell_n(\gamma,u)  \mid D,\theta^{(r)}\},\\
\beta_j^{(r+1)} & = G_{n,j}(\theta^{(r)})^{-1}n^{-1}\sum_{i=1}^n Z_i \int_{W_j} \,\dd N_i(s),
\end{split}
\notag
\end{equation}    
until $\norm{\theta^{(r+1)} - \theta^{(r)}} \leq \varepsilon$, where $G_{n,j}(\theta)$ for $j = 1,\ldots,K$ are defined in~\eqref{eq::betahat2}.
To see that this algorithm is in effect a Newton--Rhapson algorithm, denote the complete data estimating equation by, (which is a $qK + p$ dimensional column vector)
\begin{equation}
\Upsilon_n(\theta,D,U) = 
(\Upsilon_{n,1}(\theta,D,U)^{\prime},\ldots,
\Upsilon_{n,K}(\theta,D,U)^{\prime},
\Upsilon_{n,K+1}(\theta,D,U)^{\prime}
)^{\prime}
= 0.
\notag
\end{equation}   
These are complete data estimating equations in the sense that they treat the $U_1,\ldots,U_n$ as observables. Using that $U_i \delta_i = \delta_i$, we can write
\begin{equation}
\begin{split}
& \Upsilon_{n,j}(\theta,D,U) = \frac{1}{n}\sum_{i=1}^n Z_i(\delta_i I_{W_j} - U_ir_j(T_i)Z_i^{\prime}\beta_j),\quad \text{for $j = 1,\ldots,K$},\\
& \Upsilon_{n,K+1}(\theta,D,U)  = \frac{1}{n}\sum_{i=1}^n X_i(U_i - \pi_i).
\end{split}
\notag
\end{equation}
Notice that $\Upsilon_{n,j}(\theta,D,U) = \Upsilon_{n,j}(\beta,D,U)$ for $j = 1,\ldots,K$ and that $\Upsilon_{n,K+1}(\theta,D,U) = \Upsilon_{n,K+1}(\gamma,D,U)$. Let 
\begin{equation}
\dot{\Upsilon}_{n}(\theta,D,U) = \frac{\partial}{\partial\theta}\Upsilon_{n}(\theta,D,U)
= 
\begin{pmatrix}
Q_{n}(\beta,D,U) & 0\\
0 & V_{n}(\gamma,U)
\end{pmatrix} ,
\notag
\end{equation}
and note that $Q_{n}(\beta,D,U)$ is a block diagonal matrix whose blocks are $\{Q_{n}(\beta,D,U)\}_{j,j} = -n^{-1}\sum_{i=1}^nZ_iZ_i^{\prime}U_i r_{j}(T_i)$ for $j =1,\ldots,K$, while $V_{n}(\gamma,u) = -n^{-1}\sum_{i=1}^nX_iX_i^{\prime}\pi_i( 1 - \pi_i)$, is the Hessian of $\ell_n(\gamma,U)$. The observed data estimating equations given in~\eqref{eq::the_est_eq} can be expressed in terms of the complete data estimating equation as
\begin{equation}
\Psi_{n}(\theta,D) = \Upsilon_n(\theta,D,\E\{U \mid D,\theta\}) = 0.
\notag
\end{equation}
Note also that $\Upsilon_n(\theta,D,\E\{U\mid D,\theta\}) = \E\,\{\Upsilon_n(\theta,D,U)\mid D,\theta \}$. It can then be verified (this fact follows more or less directly from the fact that the inverse of a block diagonal matrix is the matrix with the inverses of its diagonal) that the sequence $\theta^{(r)}$ produced by~\eqref{eq::the_alg1} is the same as the Newton--Raphson sequence 
\begin{equation}
\theta^{(r+1)} = \theta^{(r)} - \dot{\Upsilon}_{n}(\theta^{(r)},D,\E\,\{U\mid D,\theta^{(r)}\})^{-1} \Upsilon(\theta^{(r)},D,\E\,\{U\mid D,\theta^{(r)}\}),
\notag
\end{equation}
and convergence of the sequence $(\theta^{(r)})_{r \geq 1}$ ought to follow more or less directly from, e.g., \citet[Th.~8.1.7 and Th.~8.1.8, p.~145--146]{ortega1990numerical}. 


\section{A theorem and proof of Lemma~\ref{lemma::uniqueness}}\label{app::lemma::uniqueness}
Recall that $\Delta\xi(t) = \xi(t) - \xi(t-)$ is the jump of a process $\xi(t)$ at time $t$. The next theorem is used in many of the proofs. 

\begin{theorem}\label{theorem::counting_clt} For $n = 1,2,\ldots$ let $\Xi^n(s,x) = \Xi^n(s,x_1,\ldots,x_p)$ be functions with values in $\real^p,\,p \geq 1$, and assume that $\norm{\Xi^n(s,x)} \leq b < \infty$ for all $s$, $x$ and $n$. Let $X_1,\ldots,X_n$ be i.i.d.~replicates of $X$, where $X$ is a covariate satisfying Assumption~\ref{assumption3}. Suppose that $M_i(t) = N_i(t) - \int_0^{t}Y_i(s)\alpha(s,X_i)\,\dd s$, for $i = 1,\ldots,n$, are independent counting process martingales observed over $[0,\tau]$, and set 
\begin{equation}
\xi_n(t) = n^{-1/2}\sum_{i=1}^n \int_0^t \Xi^n(s,X_i)\,\dd M_i(s). 
\notag
\end{equation}
Suppose there is a function $H(s,x) = \Xi(s,x) \sqrt{y(s;x)\alpha(s,x)}$, such that
\begin{equation}
\langle \xi_n,\xi_n\rangle_t \overset{p}\to \E\,\langle \xi,\xi\rangle_t = \E\, \int_0^t H(s,X)H(s,X)^{\prime}\,\dd s,\quad\text{for all $0 < t \leq \tau$},
\label{eq::angle_bracket_convergence}
\end{equation}
as $n \to \infty$. Then $\xi_n \Rightarrow \xi$, where 
\begin{equation}
\xi(t) = \int_0^t H(s,X)\,\dd W(s), 
\notag
\end{equation}
with $W(s)$ a one-dimensional Wiener process.
\end{theorem}
\begin{proof} Write $\xi_n = (\xi_{n,1},\ldots,\xi_{n,p})^{\prime}$, where $\xi_{n,l}$ for $l = 1,\ldots,p$ are one-dimensional square integrable local martingales. Tightness of $\xi_n$ follows from~\citet[Theorem~VI.4.13, p.~358]{jacod2003limit} because the process $Q_{n}(t) = \sum_{l =1}^{p}\langle \xi_{n,l},\xi_{n,l}\rangle_t$ is $C$-tight~\citep[Def.~VI.3.25, p.~351]{jacod2003limit}. To see that $Q_n(t)$ is $C$-tight, note that both $Q_n(t)$ and its probability limit as $n\to \infty$, say $Q(t)$, are increasing processes, and $t\mapsto Q(t)$ is continuous, which by~\citet[Theorem~VI.3.37, p.~354]{jacod2003limit} yields process convergence of $Q_n$ to $Q$, and hence $C$-tightness of $Q_n$. From tightness of $\xi_n$ it follows that for any subsequence $n_k$, we can find a further subsequence $n_{k_l}$ such that $\xi_{n_{k_l}}$ converges in law to $\xi$, say \citep[Theorem~2.6, p.~20]{billingsley1999conv}. Since no two counting processes jump at the same time $\norm{\Delta\xi_n(t)} = \norm{n^{-1/2} \sum_{i=1}^n\Xi(t,X_i)\Delta N_i(t)} \leq n^{-1/2}\max_{i \leq n}\,\norm{\Xi(t,X_i)}\leq n^{-1/2} b \leq b$. The fact that $\norm{\Delta\xi_n(t)} \leq b$ entails that $\xi$ is a local martingale with respect to the filtration that it generates \citep[Corollary~IX.1.19, p.~527]{jacod2003limit}. The tightness of $\xi_n$ combined with $\sup_{0 \leq t \leq \tau}\norm{\Delta\xi_n(t)} \leq n^{-1/2}b $, gives that $\xi_n$ is $C$-tight, hence $\xi$ is continuous \citep[Proposition~VI.3.26, p.~351]{jacod2003limit}. Moreover, that the jumps of $\xi_n$ are uniformly bounded for all $n$, combined with $\langle \xi_n,\xi_n\rangle_t$ being tight for all $t$, gives that $\xi_n$ is P-UT \citep[Proposition~6.13, p.~379]{jacod2003limit}. From $\xi_n$ being P-UT we get joint convergence of $(\xi_{n_{k_l}},[\xi_{n_{k_l}},\xi_{n_{k_l}}])$ to $(\xi,[\xi,\xi])$ \citep[Theorem~VI.6.26, p.~384]{jacod2003limit}, but $[\xi,\xi] = \langle \xi,\xi\rangle $ since $\xi$ is continuous. From~\eqref{eq::angle_bracket_convergence} this means that $\langle \xi ,\xi\rangle_t = \int_0^t H(s,X)H(s,X)^{\prime}\,\dd s$. Since $\xi$ is a continuous local martingale and $\langle \xi,\xi\rangle_t$ is absolutely continuous with respect to Lebesgue measure, we can realise $\xi$ as $\xi(t) = \int_0^t H(s,X)\,\dd W(s)$ in terms on a one-dimensional Wiener process $W$ \citep[Proposition~II.7.12, p.~133]{jacod2003limit}. Since the subsequence $n_{k_l}$ was arbitrary, the claim of the theorem follows (see the corollary on p.~337 in \citet{billingsley1995conv}).
\end{proof}
\begin{remark} The boundedness assumption on $\Xi(s,x)$ and the fact that the $M_i(s)$ are counting process martingales ensure that $\sup_{0 \leq t \leq \tau}\norm{\xi_n(t)} \leq b < \infty$ for all $n$. This bound, in turn, implies that $\int_{\abs{x} > \eps} \abs{x}^2 \nu^n([0,\tau],\dd x) \to_p 0$ for all $\eps > 0$, where $\nu^n$ is the compensator of the jump measure of $\xi^n$ (see \citet[Corollary~B.4(ii)$^{\prime\prime}$]{stoltenberg2020kappa}). Therefore, Theorem~\ref{theorem::counting_clt} can be stated without an explicit Lindeberg type condition.  
\end{remark}

Here is the proof of Lemma~\ref{lemma::uniqueness}. As above, $r_j(s) = \int_{W_j} I\{s \geq u\}\,\dd u$. 
\begin{proof}{\sc(of Lemma~\ref{lemma::uniqueness})} Notice that $r_j(s) = 0$ for $s \leq v_{j-1}$; for $s \in W_j$ it is $r_j(s) = s - v_{j-1}$; and for $s \geq v_j$, $r_j(s) = \Delta_j$. Assume that $\theta$ is such that $\Psi_{\star}^K(\theta) = 0$, then for $1 \leq j \leq K-1$,
\begin{equation}
\begin{split}
0 & = \{\Psi_{\star}^K(\theta)\}_{j}  = \E\,Z\int_0^{\tau} g_j(s,X,Z,\theta) y(s;X,Z)\{w(s,\theta^{\star})Z^{\prime}\beta^{\star}(s) - w(s,\theta)Z^{\prime}\beta(s)\}\,\dd s\\
& = \sum_{\ell=j}^K \,\E\,Z\int_{W_{\ell}} [I_{W_j}(s) - \{1 - w(s,\theta)\}r_j(s)z^{\prime}\beta_j] y(s;X,Z)\{w(s,\theta^{\star})Z^{\prime}\beta_{\ell}^{\star} - w(s,\theta)Z^{\prime}\beta_{\ell}\}\,\dd s\\
& = \E\,Z\int_{W_j}y(s;X,Z)\{w(s,\theta^{\star})Z^{\prime}\beta_j^{\star} - w(s,\theta)Z^{\prime}\beta_j\}\,\dd s\\
& \qquad \qquad - \E\,Z\sum_{\ell = j+1}^K\int_{W_{\ell}} \{1 - w(s,\theta)\}\Delta_jZ^{\prime}\beta_j y(s;X,Z)\{w(s,\theta^{\star})Z^{\prime}\beta_{\ell}^{\star} - w(s,\theta)Z^{\prime}\beta_{\ell}\}\,\dd s.
\end{split}
\notag
\end{equation}
Consider only the intercept term of $Z = (1,Z_1,\ldots,Z_{q-1})^{\prime}$ and of $X = (1,X_1,\ldots,X_{p-1})^{\prime}$, and use that $\{\Psi_{\star}^K(\theta)\}_{K+1}=0$, then for $j = 1$,
\begin{equation}
\begin{split}
0 & = \{\Psi_{\star}^K(\theta)\}_{1,1}  = \E\,\int_{W_1}y(s;X,Z)\{w(s,\theta^{\star})Z^{\prime}\beta_1^{\star} - w(s,\theta)Z^{\prime}\beta_1\}\,\dd s\\
& \qquad \qquad - \Delta_1\,\E\,Z^{\prime}\beta_1 \sum_{\ell = 2}^K\int_{W_{\ell}} \{1 - w(s,\theta)\}y(s;X,Z)\{w(s,\theta^{\star})Z^{\prime}\beta_{\ell}^{\star} - w(s,\theta)Z^{\prime}\beta_{\ell}\}\,\dd s\\
& = \E\,\int_{W_1}y(s;X,Z)\{w(s,\theta^{\star})Z^{\prime}\beta_1^{\star} - w(s,\theta)Z^{\prime}\beta_1\}\,\dd s\\
& \qquad \qquad + \Delta_1\,\E\,Z^{\prime}\beta_1 \int_{W_{1}} \{1 - w(s,\theta)\}y(s;X,Z)\{w(s,\theta^{\star})Z^{\prime}\beta_{1}^{\star} - w(s,\theta)Z^{\prime}\beta_{1}\}\,\dd s,
\end{split}
\notag
\end{equation}
where we use Assumption~\ref{adhoc_assumption} to get the third equality. Because both $y(s;x,z)$ and $1 - w(s,\theta)$ are positive functions, and because the function $w(s,\theta^{\star})Z^{\prime}\beta_{1}^{\star} - w(s,\theta)Z^{\prime}\beta_{1}$ cannot change sign on $W_1$ when both hazards are constant over $W_1$, we must have that 
\begin{equation}
\E\,Z^{\prime}\{w(s,\theta^{\star})\beta_{1}^{\star} - w(s,\theta)\beta_{1} \}
= \E\,\pi(X^{\prime}\gamma^{\star} - Z^{\prime}\beta_1^{\star}s)Z^{\prime}\beta_{1}^{\star} - \pi(X^{\prime}\gamma - Z^{\prime}\beta_1s)Z^{\prime}\beta_{1}\} = 0,
\notag
\end{equation} 
for all $s \in W_1$. Since the functions on the right are not linear in $s$, this equality may only hold if 
\begin{equation}
X^{\prime}\gamma^{\star} - Z^{\prime}\beta_1^{\star}s = X^{\prime}\gamma - Z^{\prime}\beta_1s
=  X^{\prime}(\gamma^{\star} - \gamma) - Z^{\prime}(\beta_1^{\star} - \beta_1)s = 0,
\notag
\end{equation}
almost surely with respect to the distribution of the covariates. But because the covariates are linearly independent by Assumption~\ref{assumption3}, this implies that $\gamma^{\star} = \gamma$ and $\beta_1^{\star} = \beta_1$.
\end{proof}


\section{Some notes on the $\dot{\Psi}$ matrix}\label{app::Psi_dot} 
The vector valued function $\Psi_{\star}^K = \{(\Psi_{\star,1}^K)^{\prime},\cdots,(\Psi_{\star,K}^K)^{\prime},(\Psi_{\star,K+1}^K)^{\prime}\}^{\prime}$ is the $(qK + p)$ column vector defined in~\eqref{eq::Psi_avg1}, restated here
\begin{equation}
\Psi_{\star}^K(\theta) \coloneqq \E_{\theta^{\star}}\, \Psi_{n}^K(\theta) = \E\,\int_0^{\tau} h(s,X,Z,\theta) y(s;X,Z)\{w(s,\theta^{\star})Z^{\prime}\beta^{\star}(s) - w(s,\theta)Z^{\prime}\beta(s)\}\,\dd s.
\notag
\end{equation}
Define $f_{\theta}(t) = w(t,\theta)Z^{\prime}\beta(s)$, where the locally constant function $\beta(s) = \sum_{j=1}^K \beta_j I_{W_j}(s)$. We now derive an expression for $\dot{\Psi}_{\theta}^K$. The functions $g_j(s,x,z,\theta)$ are defined in~\eqref{eq::small_g_func}. For $j = 1,\ldots,K$,
\begin{equation}
\begin{split} 
\frac{\partial\Psi_{\star,j}^{K}(\theta)}{\partial\beta_j}& = - \E_{\theta^{\star}}ZZ^{\prime}\int_0^{\tau} y(s;X,Z)w(s,\theta)g_j(s,X,Z,\theta)[I_{W_j}(s) - \{1 - w(s,\theta)\}r_j(s)Z^{\prime}\beta(s)]\,\dd s\\
& \qquad \qquad - \E_{\theta^{\star}}ZZ^{\prime}\int_0^{\tau} y(s;X,Z)\{1 - w(s,\theta)\} r_j(s)\{f_{\theta^{\star}}(s) - f_{\theta}(s)\}\,\dd s\\ 
& \qquad \qquad - \E_{\theta^{\star}}ZZ^{\prime}\int_0^{\tau} y(s;X,Z)w(s,\theta)\{1 - w(s,\theta)\}Z^{\prime}\beta_j r_j(s)^2\{f_{\theta^{\star}}(s) - f_{\theta}(s)\}\,\dd s
\end{split}
\notag
\end{equation} 
The first set of cross derivatives are, for $j = 1,\ldots,K$,  
\begin{equation}
\begin{split}
\frac{\partial\Psi_{\star,j}^{K}(\theta)}{\partial\beta_{\ell}} & = 
-\E_{\theta^{\star}}ZZ^{\prime}\int_0^{\tau}y(s;X,Z)w(s,\theta)g_j(s,X,Z,\theta)[I_{W_{\ell}}(s) - \{1 - w(s,\theta)\}r_{\ell}(s)Z^{\prime}\beta(s)]\,\dd s\\
& \qquad \quad - \E_{\theta^{\star}}ZZ^{\prime}\int_0^{\tau}y(s;X,Z)w(s,\theta)\{1 - w(s,\theta)\}r_{\ell}(s)r_j(s)Z^{\prime}\beta_j \{f_{\theta^{\star}}(s) - f_{\theta}(s)\}\,\dd s.
\end{split}
\notag
\end{equation}
And, for $j = K+1$, (recall from~\eqref{eq::lambda_mu_nu_xi} that $\mu_{\theta}(s;x,y) = y(s;x,y)w(s,\theta)\{1 - w(s,\theta)\}$) 
\begin{equation}
\begin{split}
\frac{\partial\Psi_{\star,K+1}^{K}(\theta)}{\partial\beta_j} & =  -\E_{\theta^{\star}}\, ZX^{\prime} \int_0^{\tau}\mu_{\theta}(s;X,Y)[I_{W_j}(s) - \{1 - w(s,\theta)\}r_j(s)Z^{\prime}\beta(s) ]\,\dd s\\
& \qquad\qquad + \E_{\theta^{\star}}\,ZX^{\prime} \int_0^{\tau}y(s;X,Z) w(s,\theta)\{1 - w(s,\theta)\}r_j(s) \{f_{\theta^{\star}}(s) - f_{\theta}(s)\}\,\dd s.
\end{split}
\notag
\end{equation}
For $j = 1,\ldots,K$.
\begin{equation}
\begin{split}
\frac{\partial\Psi_{\star,j}^{K}(\theta)}{\partial\gamma} & =  
-\E_{\theta^{\star}}XZ^{\prime}\int_0^{\tau} y(s;X,Z)w(s,\theta)\{1 - w(s,\theta)\}g_j(s)Z^{\prime}\beta(s)\,\dd s \\
& \qquad \qquad +\E_{\theta^{\star}}XZ^{\prime}\int_0^{\tau}y(s;X,Z)w(s,\theta)\{1 - w(s,\theta)\}r_j(s)Z^{\prime}\beta_j\{f_{\theta^{\star}}(s) - f_{\theta}(s)\}\,\dd s,
\end{split}
\notag
\end{equation}
and, for $j = K+1$,
\begin{equation}
\begin{split}
\frac{\partial\Psi_{\star,K+1}^{K}(\theta)}{\partial\gamma} & =  -\E_{\theta^{\star}}\, XX^{\prime} \int_0^{\tau}y(s;X,Z)w(s,\theta)\{ 1 - w(s,\theta)\}^2 Z^{\prime}\beta(s)\,\dd s\\
& \qquad \qquad -\E_{\theta^{\star}}\, XX^{\prime}\int_0^{\tau} y(s;X,Z)w(s,\theta)\{1 - w(s,\theta)\}\{f_{\theta^{\star}}(s) - f_{\theta}(s)\}\,\dd s. 
\end{split}
\notag
\end{equation}
This gives the matrix
\begin{equation}
\begin{split}
\dot{\Psi}_{\theta}^K = 
\begin{pmatrix}
\dot{\Psi}_{\theta,00}^K & \dot{\Psi}_{\theta,01}^K\\
\dot{\Psi}_{\theta,10}^K & \dot{\Psi}_{\theta,11}^K,
\end{pmatrix},
\end{split}
\notag
\end{equation}
where, 
\begin{equation}
\begin{split}
\dot{\Psi}_{\theta,00}^K & = 
- 
\begin{pmatrix}
\frac{\partial}{\partial\beta_1}\Psi_{\star,1}^{K}(\theta) & \cdots & \frac{\partial}{\partial\beta_1}\Psi_{\star,K}^{K}(\theta)\\
\vdots	& \ddots & \vdots\\
\frac{\partial}{\partial\beta_K}\Psi_{\star,1}^{K}(\theta)& \cdots & \frac{\partial}{\partial\beta_K}\Psi_{\star,K}^{K}(\theta)
\end{pmatrix},
\qquad
\dot{\Psi}_{\theta,01}^K = -
\begin{pmatrix}
\frac{\partial}{\partial\beta_1}\Psi_{\star,K+1}^{K}(\theta)\\
\vdots \\
\frac{\partial}{\partial\beta_K}\Psi_{\star,K+1}^{K}(\theta)
\end{pmatrix},
\end{split}
\notag
\end{equation}
and
\begin{equation}
\dot{\Psi}_{\theta,10}^K  = 
- 
\begin{pmatrix}
\frac{\partial}{\partial\gamma}\Psi_{\star,1}^K(\theta)\\
\vdots \\
\frac{\partial}{\partial\gamma}\Psi_{\star,K}^K(\theta)
\end{pmatrix}^{\prime},\qquad 
\dot{\Psi}_{\theta,11}^K = -\frac{\partial}{\partial\gamma}\Psi_{\star,K+1}^{K}(\theta) .
\notag
\end{equation}

\begin{lemma}\label{lemma::dotpsi_limit} When $\dot{\Psi}_{\theta}^K$ is evaluated in $\theta^{\star}$, 
\begin{equation}
\begin{split}
\{\dot{\Psi}_{\theta,00}^K\}_{j,j} & = -
\frac{\partial}{\partial\beta_j} \Psi_{\star,j}^{K}(\theta^{\star})  = G_j(\theta^{\star}) + O(\Delta_j^2),\\
\{\dot{\Psi}_{\theta,00}^K\}_{j,\ell} & = -\frac{\partial}{\partial\beta_{\ell}} \Psi_{\star,j}^{K}(\theta^{\star})  = O(\Delta_j^2\vee\Delta_{\ell}^2),\\
\{\dot{\Psi}_{\theta,01}^K\}_{j} & = -\frac{\partial}{\partial\beta_{j}} \Psi_{\star,K+1}^{K}(\theta^{\star})  = \E\,ZX^{\prime}\int_{W_j}y(s;X,Z)w(s,\theta^{\star})\{1 - w(s,\theta^{\star})\}\,\dd s + O(\Delta_j^2),\\
\{\dot{\Psi}_{\theta,10}^K\}_{j} & = -\frac{\partial}{\partial\gamma} \Psi_{\star,j}^{K}(\theta^{\star})  
= \E\,XZ^{\prime}\int_{W_j}y(s;X,Z)w(s,\theta^{\star})\{1 - w(s,\theta^{\star})\} Z^{\prime}\beta^{\star}_j\,\dd s + O(\Delta_j^2),\\
\dot{\Psi}_{\theta,11}^K &= - \frac{\partial}{\partial\gamma} \Psi_{K+1}^{K}(\theta^{\star})  = \E\,XX^{\prime}\int_0^{\tau} y(s;X,Z)w(s,\theta^{\star})\{1 - w(s,\theta^{\star})\}^2 Z^{\prime}\beta^{\star}(s)\,\dd s,
\end{split}
\notag
\end{equation}
as $K \to \infty$.
\end{lemma}
\begin{proof} Recall that $g_j(s,X,Z,\theta) = I_{W_j}(s) - \{1 - w(s,\theta)\}r_j(s)Z^{\prime}\beta_j$ for $j =1,\ldots,K$. That $r_j(s) \leq \Delta_j$ for any $0 < s \leq \tau$ is used repeatedly. We also use that if $\alpha(s)$ is a real function with a bounded derivative $\dot{\alpha}(s)$ on $[0,\tau]$, then
\begin{equation}
\begin{split}
%
|\int_{W_j} \alpha(s) r_j(s)\,\dd s| & 
\leq \Delta_j\int_{W_j} |\alpha(s)| \,\dd s
\leq \Delta_j^2 \alpha(v_{j-1}) + \Delta_j \int_{W_j}|\alpha(s) - \alpha(v_{j-1})|\,\dd s\\
&\leq \Delta_j^2 \alpha(v_{j-1}) + \Delta_j\sup_{s \in W_j}|\dot{\alpha}(s)|\int_{W_j}|s - v_{j-1}|\,\dd s\\
& = \Delta_j^2 \alpha(v_{j-1}) + \Delta_j\sup_{s \in W_j}|\dot{\alpha}(s)|\,\frac{\Delta_j^2}{2}
= O(\Delta_j^2).
\end{split}
\notag
\end{equation} 
Evaluated in $\theta^{\star}$, 
\begin{equation}
\begin{split}
-\frac{\partial}{\partial \beta_j} \Psi_{\star,j}(\theta^{\star}) & 
= \E\, ZZ^{\prime}\int_0^{\tau} y(s;X,Z) w(s,\theta^{\star}) g_j(s,X,Z,\theta^{\star}) \\
& \qquad \qquad \qquad \qquad \times[I_{W_j}(s) - \{1 - w(s,\theta^{\star})\}r_j(s) Z^{\prime}\beta^{\star}(s)]\,\dd s\\
& = 
G_j(\theta^{\star}) - 2\, \E\, ZZ^{\prime}\int_{W_j}y(s;X,Z) w(s,\theta^{\star})\{1 - w(s,\theta^{\star})\}r_j(s)Z^{\prime}\beta_j^{\star}\,\dd s \\
& \qquad \qquad   
+ \E\, ZZ^{\prime}\int_0^{\tau} y(s;X,Z) w(s,\theta^{\star})\{1 - w(s,\theta^{\star})\}^2r_j(s)^2z^{\prime}\beta_j^{\star}Z^{\prime}\beta^{\star}(s)\,\dd s\\
& = G_j(\theta^{\star}) + O(\Delta_j^2),
\end{split}
\notag
\end{equation}
and,
\begin{equation}
\begin{split}
- \frac{\partial}{\partial\beta_{\ell}}\Psi_{\star,j}^{K}(\theta^{\star}) & = 
\E\,ZZ^{\prime}\int_0^{\tau}y(s;X,Z)w(s,\theta^{\star})g_j(s,X,Z,\theta^{\star})\\
& \qquad \qquad \qquad \qquad \times[I_{W_{\ell}}(s) - \{1 - w(s,\theta^{\star})\}r_{\ell}(s)Z^{\prime}\beta^{\star}(s)]\,\dd s\\
& = -\E\,ZZ^{\prime}\int_{W_j}y(s;X,Z)w(s,\theta^{\star})\{1 - w(s,\theta^{\star})\}r_{\ell}(s)Z^{\prime}\beta^{\star}(s)]\,\dd s\\
& \qquad \qquad \qquad \qquad -\E\,ZZ^{\prime}\int_{W_{\ell}}y(s;X,Z)w(s,\theta^{\star})\{1 - w(s,\theta^{\star})\}r_{j}(s)Z^{\prime}\beta_j^{\star}]\,\dd s\\
& + \E\,ZZ^{\prime}\int_{0}^{\tau}y(s;X,Z)w(s,\theta^{\star})\{1 - w(s,\theta^{\star})\}^2r_j(s)r_{\ell}(s)Z^{\prime}\beta^{\star}(s)Z^{\prime}\beta_j^{\star}]\,\dd s\\
& = O(\Delta_j^2\vee \Delta_{\ell}^2).
\end{split}
\notag
\end{equation}
By the same arguments we get the approximations for $\partial \Psi_{\star,K+1}^K(\theta^{\star})/\partial \beta_j$ and $\partial \Psi_{\star,j}^K(\theta^{\star})/\partial \gamma$, while the lower right corner $\partial \Psi_{\star,K+1}^K(\theta^{\star})/\partial \gamma$ remains unchanged.
\end{proof}
The inverse of $\dot{\Psi}_{\theta}^K$ is 
\begin{equation}
\begin{split}
(\dot{\Psi}_{\theta}^K)^{-1} 
& = 
\begin{pmatrix}
(\dot{\Psi}_{\theta,00}^{K})^{-1} + (\dot{\Psi}_{\theta,00}^{K})^{-1} \dot{\Psi}_{\theta,01}^{K} \dot{\Psi}_{\theta}^{K,11}\dot{\Psi}_{\theta,10}^{K}(\dot{\Psi}_{\theta,00}^{K})^{-1} 
& -(\dot{\Psi}_{\theta,00}^{K})^{-1} \dot{\Psi}_{\theta,01}^{K} \dot{\Psi}_{\theta}^{K,11}\\
- \dot{\Psi}_{\theta}^{K,11}\dot{\Psi}_{\theta,10}^{K}(\dot{\Psi}_{\theta,00}^{K})^{-1} & \dot{\Psi}_{\theta}^{K,11}
\end{pmatrix},
\end{split}
\notag
\end{equation}
where 
\begin{equation}
\dot{\Psi}_{\theta}^{K,11} = (\dot{\Psi}_{\theta,11}^{K} - \dot{\Psi}_{\theta,10}^{K}(\dot{\Psi}_{\theta,00}^{K})^{-1}	\dot{\Psi}_{\theta,01}^{K}	)^{-1}.
\notag
\end{equation}
From Lemma~\ref{lemma::dotpsi_limit} we have that $\Psi_{\theta^{\star},00}^K$ is approximately a block diagonal matrix, that is 
\begin{equation}
\Psi_{\theta^{\star},00}^K 
%
= \E\,
\begin{pmatrix}
ZZ^{\prime}\int_{W_1}\lambda_{\theta^{\star}}(s)\,\dd s & \cdots & 0\\
\vdots & \ddots & \vdots\\
0 & \cdots & ZZ^{\prime}\int_{W_K}\lambda_{\theta^{\star}}(s)\,\dd s
\end{pmatrix}
+ O(\max_{j \leq K}\Delta_j^2),
\notag
\end{equation}
where $G_j(\theta^{\star}) = \E\, ZZ^{\prime} \int_{W_j}\lambda_{s}(\theta^{\star})\,\dd s$ for $j = 1,\ldots,K$, which results in the nice form of the following matrix,  
\begin{equation}
\begin{split}
& (\Psi_{\theta^{\star}}^{K,11})^{-1}  = 
\int_{0}^{\tau}\E\,XX^{\prime}\xi_{\theta^{\star}}(s)\,\dd s \\
&\; - \sum_{j=1}^K \E\,\{XZ^{\prime}\int_{W_j}\nu_{\theta^{\star}}(s)\,\dd s\}
(\E\,\{ZZ^{\prime} \int_{W_j}\lambda_{\theta^{\star}}(s)\,\dd s\})^{-1}\E\,\{ZX^{\prime}\int_{W_j}\mu_{\theta^{\star}}(s)\,\dd s\} +  O(\max_{j \leq K}\Delta_j^2).
\end{split}
\notag
\end{equation}

\begin{lemma}\label{lemma:Psi11_approx} As $K \to \infty$,
\begin{equation}
\dot{\Psi}_{\theta^{\star}}^{K,11} = \dot{\Psi}^{11} + O(\max_{j \leq K}\Delta_j).
\notag
\end{equation}  
where
\begin{equation}
(\dot{\Psi}^{11})^{-1} = \int_{0}^{\tau}\E\,XX^{\prime}\xi_{\varphi^{\circ}}(s)\,\dd s -  \int_{0}^{\tau}\E\,\{XZ^{\prime}\nu_{\varphi^{\circ}}(s)\}G(s,\varphi^{\circ})^{-1}\E\,\{ZX^{\prime}\mu_{\varphi^{\circ}}(s)\}\,\dd s, 
\notag
\end{equation}  
\end{lemma}
\begin{proof} Since $\lambda_{\theta}(s) = \lambda_{\theta}(v_{j-1})\Delta_j + O(\Delta_j^2)$, $\mu_{\theta}(s) = \mu_{\theta}(v_{j-1})\Delta_j + O(\Delta_j^2)$, and $\nu_{\theta}(s) = \nu_{\theta}(v_{j-1})\Delta_j + O(\Delta_j^2)$ for $s \in W_j$ and $j = 1,\ldots,K$, all with bounded derivatives on $(0,\tau]$,
\begin{equation}
\begin{split}
& \sum_{j=1}^K \E\,\{XZ^{\prime}\int_{W_j}\nu_{\theta^{\star}}(s)\,\dd s\}(\E\,\{ZZ^{\prime} \int_{W_j}\lambda_{\theta^{\star}}(s)\,\dd s\})^{-1}\E\,\{ZX^{\prime}\int_{W_j}\mu_{\theta^{\star}}(s)\,\dd s\}\\
& \qquad \qquad 
= \sum_{j=1}^K \E\,\{XZ^{\prime}\nu_{\theta^{\star}}(v_{j-1})\}G(v_{j-1},\theta^{\star})^{-1}\E\,\{ZX^{\prime}\mu_{\theta^{\star}}(v_{j-1})\}\Delta_j + O(\max_{j\leq K}\Delta_j), 
\end{split}
\notag
\end{equation} 
where $G(s,\theta) = \E\,ZZ^{\prime} \lambda_{\theta}(s) = \E\, ZZ^{\prime}y(s;X,Z)w(s,\theta)$ as defined in~\eqref{eq::Gfunc}, and the result follows.
\end{proof}

\section{Limit theorems for Aalen{'}s linear hazard model}\label{app::aalens_linear_pfs}
In this appendix we return to the classical setting (that is, no cured fraction, $\pi \equiv 1$) introduced in Section~\ref{sec::2.1}, and derive the limiting distribution of the Aalen{'}s linear hazard regression estimator $\widetilde{B}$ of~\eqref{eq::standard.est} in three different ways. In Section~\ref{subapp::locally_pf} we start out with a locally constant model for a continuous truth (i.e.~continuous regression function $\beta_l^{\circ}(s)$ for $l=0,\ldots,{q-1}$), derive an estimator, and then let the mesh size tend to zero. See~\citet[Section~5.3, p.~149]{hermansen2015bernshteuin} for a similar construction in the case of independent and identically distributed lifetimes. In Section~\ref{app::contiguity} we use the measure change techniques introduced in Section~\ref{sec::shifting_back}, that is, we derive our estimator under a data generating mechanism associated with a locally constant truth, and then shift back to the measure associated with a continuous truth. The third technique, employed in Section~\ref{subapp::spurver}, consists for using nonparametric likelihood theory, as well as some results for the theory of empirical processes. (see e.g.,~\citet{gill1989non,gill1993non} or \citet[Ch.~IV.1.5 and Ch.~VIII]{andersen1993statistical}). 

Recall that we have i.i.d.~survival data $(T_1,\delta_1,Z_1),\ldots,(T_n,\delta_n,Z_n)$, that are replicates of $(T,\delta,Z)$, observed over the interval $[0,\tau]$, where $T = T^*\wedge C$, the $C$ stem from an absolutely continuous distribution $H_c$. The survival times $T_i^{*}\mid Z_i$ stem from a distribution with hazard rate 
\begin{equation}
\alpha_i^{\circ}(t) = Z_i^{\prime}\beta^{\circ}(t)
= \beta_{0}^{\circ}(t) + Z_{i,1}\beta_{1}^{\circ}(t) + \cdots + Z_{i,q-1}\beta_{q-1}^{\circ}(t),
\label{eq::true_model1_2}
\end{equation}
where the functions $\beta_0^{\circ}(s),\ldots,\beta_{q-1}^{\circ}(s)$ are assumed to be continuously differentiable on $[0,\tau]$; the covariates are linearly independent, bounded with probability one, and the matrix $\E\, ZZ^{\prime}$ is positive definite. Also assume that $Z^{\prime}\beta^{\circ}(t) > 0$ for all $Z$ in the support of the distribution of the covariates. We denote the distribution associated with the $\beta_0^{\circ}(s),\ldots,\beta_{q-1}^{\circ}(s)$ by $P_{\circ}$. As above, we write $y(s) = \E\, y(s ; Z) = \E\,\E_{\varphi^{\circ}}\,\{Y(s)\mid Z\}$, and note that 
\begin{equation}
\sup_{0 \leq t\leq \tau}|n^{-1}\sum_{i=1}^n \E_{\varphi^{\circ}}\,\{Y_i(s)\mid Z\} - y(s;Z)|\overset{p}\to 0,
\notag
\end{equation} 
as $n \to \infty$.      

\subsection{Locally constant model}\label{subapp::locally_pf} 
As a model for~\eqref{eq::true_model1_2} we take the regression functions $\beta(s)$ as locally constant. More precisely, let $0 = v_0 < v_1 < \cdots < v_{K-1} < v_K = \tau$ be a partition of the interval. For $j = 1,\ldots,K$ set $W_j = [v_{j-1},v_j)$; we assume $v_j = j\tau/K$ for $j = 1,\ldots,K$; and $I_{W_j}(t) = 1$ if $t\in W_j$ and zero otherwise. Suppose that $\tau = 1$. Define the vector valued function and the model for the hazard rate by 
\begin{equation} 
\beta(t) = \sum_{j=1}^K \beta_j I_{W_j}(t), \quad \text{and}\quad
\alpha_i(t) = Z_i^{\prime} \beta(t) = Z_i^{\prime}\sum_{j=1}^K \beta_j I_{W_j}(t),
\label{eq::beta_model_classic}
\end{equation}
respectively. Here $\beta_j = (\beta_{0,j},\ldots,\beta_{q-1,j})^{\prime}$ for $j = 1,\ldots,K$ are $q$-dimensional column vectors of coefficients, so this is a model with $qK$ unknown parameters to be estimated from the data. The functions $r_{j}(t)$ are defined as in~\eqref{eq::rj_funcs}, namely, $r_{j}(t) = \int_{W_j} I\{t \geq s\} \,\dd s = \int_0^{t}I_{W_j}(s)\,\dd s$ for $j = 1,\ldots,K$. 
The likelihood function of the model in~\eqref{eq::beta_model_classic} is given by 
\begin{equation}
\ell_n(\beta) = \sum_{i=1}^{n} \{\delta_i \log \alpha_i(T_i) - \int_0^{T_i}\alpha_i(s)\,\dd s\} 
= \sum_{i=1}^n \int_0^{\tau} \{ \log \alpha_i(s)\,\dd N_i(s) - Y_i(s)\alpha_i(s)\,\dd s 			 \}.
\notag
\end{equation}
Differentiate with respect to the $j${'}th column vector $\beta_j$ to get the score functions, they are  
\begin{equation}
\begin{split}
U_{n,j}(\beta) & = \sum_{i=1}^n Z_i\big\{ \frac{\delta_iI_{W_j}(T_i)}{\alpha_i(T_i)}
- I_{W_j}(T_i)	\big\}
= \sum_{i=1}^n Z_i \frac{I_{W_j}(T_i)}{Z_i^{\prime}\beta_j} \big\{\delta_i - Z_i^{\prime}\beta_j\big\},\quad\text{for $j = 1,\ldots,K$}.
\end{split}
\notag
\end{equation}
This motivates the $Z$-estimators $\widehat{\beta}_1,\ldots,\widehat{\beta}_K$, defined as the zeros of 
\begin{equation}
\begin{split}
\Psi_{n,j}(\beta) & = \frac{1}{n}\sum_{i=1}^n Z_i I_{W_j}(T_i) \big\{\delta_i - Z_i^{\prime}\beta_j\big\}\\
& = \frac{1}{n}\sum_{i=1}^n Z_i \int_{W_j}\{\dd N_i(s) - Y_i(s) Z_i^{\prime}\beta_j\,\dd s\},
\quad \text{for $j = 1,\ldots ,K$}.
\end{split}
\label{eq::Z_esteq1}
\end{equation}
The equations~\eqref{eq::Z_esteq1} are similar to those that in the nonparametric case lead to the Aalen linear hazard estimator. We see that the $\widehat{\beta}_j$ are given by
\begin{equation}
\widehat{\beta}_j = G_{n,j}^{-1}\,\frac{1}{n}\sum_{i=1}^n Z_i\int_{W_j}\,\dd N_i(s),\quad\text{with}\quad 
G_{n,j} = \frac{1}{n}\sum_{i=1}^n Z_iZ_i^{\prime} r_j(T_i),
\notag
\end{equation}  
for $j = 1,\ldots,K$. Denote the asymptotic versions of the functions $\Psi_{n,j}(\beta)$ by $\Psi_{j}(\beta)$. These are given by 
\begin{equation}
\begin{split}
\Psi_{j}(\beta) 
& = \E_{\varphi^{\circ}}\, Z \int_{W_j}y(s ; Z) Z^{\prime}\{\beta^{\circ}(s) - \beta_j\}\,\dd s. 
\quad j = 1,\ldots ,K.
\end{split}
\notag
\end{equation}
The zeros of $\Psi_{j}(\beta)$, which we refer to as the {\it least-false} parameter values and denote by $\beta_j^{\rm lf}$, are given by, 
\begin{equation}
\beta_j^{\rm lf} = G_j^{-1}\,\E\, Z \int_{W_j} y(s ; Z) Z^{\prime}\beta^{\circ}(s)\,\dd s,\quad\text{for $j = 1,\ldots,K$}, 
\notag
\end{equation} 
where $G_j = \E\, ZZ^{\prime} \int_{W_j} y(s ; Z)\,\dd s$. Since $n^{-1}\sum_{i=1}^nZ_i\int_{W_j}\dd N_i(s) = \E\, Z\int_{W_j}y(s;Z)Z^{\prime}\beta^{\circ}(s)\,\dd s + o_p(1)$ and $G_{n,j}= G_j + o_p(1)$ for $j =1,\ldots,K$, it follows from the Cram{\'e}r--Slutsky rules that  
\begin{equation}
\widehat{\beta}_j = \beta_j^{\rm lf} + o_p(1),\quad\text{for $j = 1,\ldots,K$},
\notag
\end{equation}
as $n \to \infty$, and the mesh size is held constant. Let $j_{\max} = j_{\max}(t) = \max\{j : v_j < t\}$, then 
\begin{equation}
\begin{split}
\widehat{B}(t) - B^{\rm lf}(t) & = \frac{1}{K}\sum_{j : v_j \leq t}  (\widehat{\beta}_j -\beta_j^{\rm lf}) + (t - v_{j_{\max}})(\widehat{\beta}_{j_{\max}+1} - \beta_{j_{\max}+1}^{\rm lf}) + o_p(1), 
\end{split}
\label{eq::lf_consistency}
\end{equation}
as $n \to \infty$, and we see that this holds independently of the mesh size. 
\begin{lemma}\label{lemma::prelim_consistency} As $n \to \infty$ and $K \to \infty$
\begin{equation}
\sup_{0 \leq t \leq \tau}\norm{\widehat{B}(t) - B^{\circ}(t)} \overset{p}\to 0. 
\notag
\end{equation}
\end{lemma}
\begin{proof} Since we already have~\eqref{eq::lf_consistency}, it suffices to prove that $\sup_{0 \leq t \leq \tau}\norm{B^{\rm lf}(t) - B^{\circ}(t)}\to 0$. But for $s \in W_j$, $\beta_j^{\rm lf} - \beta^{\circ}(s) = O(1/K)$, because 
\begin{equation}
\begin{split}
\norm{\beta_j^{\rm lf} - \beta^{\circ}(s)} & = \norm{G_{j}^{-1}\E\,z\int_{W_j}y(u;z) z^{\prime}\beta^{\circ}(u)\,\dd u - \beta^{\circ}(s)}  \\
& = \norm{G_{j}^{-1}\E\,z\int_{W_j}y(u;z) z^{\prime}\{\beta^{\circ}(u) - \beta^{\circ}(s)\}\,\dd u}\\
& \leq \sup_{u \in W_j}\norm{\beta^{\circ}(u) - \beta^{\circ}(s)}
\leq \sup_{u \in W_j}\norm{\dot{\beta}^{\circ}(u)}|u - s| \leq \sup_{u \in W_j}\frac{\norm{\dot{\beta}^{\circ}(u)}}{K},
\end{split}
\notag
\end{equation}
where we use that $G_j^{-1}\E\, zz^{\prime}\int_{W_j}y(u;z)\,\dd u = I_q$. Then 
\begin{equation}
B^{\rm lf}(t) - B^{\circ}(t) 
= \sum_{j : v_j \leq t}^K\int_{W_j}\{\beta^{\rm lf}(s) - \beta^{\circ}(s)\}I_{W_j}(s)\,\dd s
= O(1/K),
\notag
\end{equation}
for all $t$, which combined with~\eqref{eq::lf_consistency} proves the lemma.
\end{proof}
  
\begin{prop}\label{theorem::clt1} The sequence $\sqrt{n}(\widehat{B} - B^{\circ})$ converges weakly to $\Uscr$ under $P_{\circ}$ as $n,K \to \infty$, provided $\sqrt{n}/K \to 0$; where $\Uscr$ is a mean zero Gaussian martingale with
\begin{equation}
\langle \Uscr,\Uscr\rangle_t = \E\,\int_{0}^{t} G(s)^{-1}Z y(s; Z) Z^{\prime} \beta^{\circ}(s)Z^{\prime} G(s)^{-1} \,\dd s,
\quad\text{and}\quad G(s) = \E\, ZZ^{\prime} y(s;Z).
\notag
\end{equation}
\end{prop}  
\begin{proof} Let $\mathscr{U}_n^{\rm lf}(t) = \sqrt{n}\{\widehat{B}(t) - B^{\rm lf}(t)\}$. From the proof of Lemma~\ref{lemma::prelim_consistency}, 
\begin{equation}
\sqrt{n}\{\widehat{B}(t) - B^{\circ}(t)\} = \sqrt{n}\{\widehat{B}(t) - B^{\rm lf}(t)\} 
+ O(\sqrt{n}/K) = \mathscr{U}_n^{\rm lf}(t) + O(\sqrt{n}/K),
\label{eq::decompose_Bhat_Blf}
\end{equation}
From the Cram{\'e}r--Slutsky rules (or, alternatively, use that $\Psi_{n,j}(\beta)$ in~\eqref{eq::Z_esteq1} is Lipschitz; that $-\partial\Psi_{j}(\beta)/\partial \beta_j = G_j$, and $\partial\Psi_{j}(\beta)/\partial \beta_l = 0$ for $l \neq j$; and that the block diagonal matrix with $G_1,\ldots,G_K$ on its diagonal is invertible because $\E\,ZZ^{\prime}$ is positive definite, then appeal to~\citet[Theorem~5.21, p.~52]{vanderVaart1998}) we have that $\widehat{\beta}_j - \beta_j^{\rm lf} = G_j^{-1}n^{-1}\sum_{i=1}^nZ_i\int_{W_j}\,\dd M_i(s) + o_p(1/\sqrt{n})$, as $n \to \infty$, so that,
\begin{equation}
\mathscr{U}_n^{\rm lf}(t) = \sqrt{n}\{\widehat{B}(t) - B^{\rm lf}(t)\} = \frac{1}{\sqrt{n}K}\sum_{i=1}^n\sum_{j: v_j \leq t} G_j^{-1}Z_i\int_{W_j}\,\dd M_i(s) + o_p(1),
\label{eq::Z_j_K_MGrep}
\end{equation}
as $n \to \infty$. The quadratic variation of this process is 
\begin{equation}
\begin{split}
\langle \mathscr{U}_n^{\rm lf},\mathscr{U}_n^{\rm lf}\rangle_t & 
= \frac{1}{nK^2}\sum_{i=1}^n\sum_{j: v_j \leq t} G_j^{-1}Z_i\int_{W_j}Y_i(s)Z_i^{\prime}\beta^{\circ}(s)\,\dd s Z_i^{\prime}G_j^{-1} + o_p(1)\\
& = \frac{1}{K^2}\sum_{j: v_j \leq t}G_j^{-1}\E\,Z\int_{W_j}y(s;Z)z^{\prime}\beta^{\circ}(s)\,\dd s\,Z^{\prime}G_j^{-1} + o_p(1).
\end{split}
\notag
\end{equation}
Because $y(s;z)$ and $\beta^{\circ}(s)$ have bounded derivatives, $G_j = G(v_{j-1})/K + O(1/K^2)$ for $j = 1,\ldots,K$, and $\int_{W_j}y(s;Z)Z^{\prime}\beta^{\circ}(s)\,\dd s = y(v_{j-1};Z)Z^{\prime}\beta^{\circ}(v_{j-1})/K + O_p(1/K^2)$, we find
\begin{equation}
\langle \mathscr{U}_n^{\rm lf},\mathscr{U}_n^{\rm lf}\rangle_t
= \frac{1}{K}\sum_{j : v_{j}\leq t}
G(v_{j-1})^{-1} \E\, Z y(v_{j-1};Z)Z^{\prime}\beta^{\circ}(v_{j-1}) Z^{\prime} G(v_{j-1})^{-1} + O(1/K) + o_p(1),
\notag
\end{equation}  
as $n,K \to \infty$, which is a Riemann sum converging to $\langle \Uscr,\Uscr\rangle_t$ as $K\to \infty$. Due to~\eqref{eq::decompose_Bhat_Blf}, this means that $n\langle\widehat{B} - B^{\circ} ,\widehat{B} - B^{\circ}\rangle_t \to_p \langle \Uscr,\Uscr\rangle_t$ for all $t$, as $n\to \infty$ and $K \to \infty$, as long as $\sqrt{n}/K \to 0$. Write $\sqrt{n}(\widehat{B}(t) - B^{\circ}(t)) = n^{-1/2}\sum_{i=1}^n \int_0^t \Xi(s,Z_i)\,\dd M_i(s) + O_p(\sqrt{n}/K) + o_p(1)$, where $\Xi(s,Z_i) = K^{-1}\sum_{j=1}^K G_j^{-1}Z_i I_{W_j}(s)$. Since $G_j^{-1}$ is positive definite, $\norm{\Xi(s,Z_i)}^2 \leq \max_{j \leq K} \norm{G_j^{-1}Z_i}^2 \leq \max_{j \leq K} \lambda_{\max,j}^2\norm{Z_i}^2$, where $\lambda_{\max,j}$ is the largest eigenvalue of $G_j^{-1}$. By assumption, $\norm{Z_i}$ is bounded with probability one for all $i$. The result then follows from Theorem~\ref{theorem::counting_clt} in Appendix~\ref{app::lemma::uniqueness}.   

\end{proof}
The limiting process $\Uscr$ of $\sqrt{n}(\widehat{B} - B^{\circ})$ of Theorem~\ref{theorem::clt1} is, not surprisingly, the same as that of the Aalen linear hazard estimator presented in~\eqref{eq::aalen_estimator_clt}.

\subsection{The limiting distribution by contiguity}\label{app::contiguity}
In this section we use the techniques employed in Section~\ref{sec::shifting_back} to derive the limiting distribution of $\sqrt{n}(\widehat{B} - B^{\circ})$ under $P_{\circ}$. We retain the assumptions of the preceding section, except that we now assume that the data are generated by the distributions $(P_{\star}^K)_{K}$, associated with the regression function $\beta^{\star}(s) = \sum_{j=1}^{K}\beta_j^{\star}I_{W_j}(s)$, with $\beta_j^{\star} = \beta^{\circ}(v_{j-1})$ for each $j = 1,\ldots,K$. This means that,  
\begin{equation}
M_i(t,\beta^{\star}) = N_i(t) - \int_0^t Y_i(s)Z_i^{\prime}\beta^{\star}(s)\,\dd s,\quad\text{for $i = 1,\dots,n$}, 
\notag
\end{equation}
are martingales under $P_{\star}^K$. We start by studying $\sqrt{n}(\widehat{B} - B^{\star})$ under $P_{\star}^K$. By an argument similar to that employed to obtain~\eqref{eq::Z_j_K_MGrep}, we now have that 
\begin{equation}
\sqrt{n}\{\widehat{B}(t) - B^{\star}(t)\} = \frac{1}{\sqrt{n}K}\sum_{i=1}^n\sum_{j: v_j \leq t} G_j^{-1}Z_i\int_{W_j}\,\dd M_i(s,\beta^{\star}) + o_p(1),
\notag
\end{equation}
as $n \to \infty$, and we see that $\sup_{0 \leq t\leq \tau}\norm{\widehat{B}(t) - B^{\star}(t)} \to_p 0$. By a proof that is essentially the same as that of Theorem~\ref{theorem::clt1}, the sequence $\sqrt{n}(\widehat{B} - B^{\star})$ converges weakly to $\Uscr$ under $(P_{\star}^K)_{K}$ as $n,K \to \infty$, but now we do not need that $\sqrt{n}/K \to 0$. Here $\Uscr$ is a mean zero Gaussian martingale, whose quadratic variation $\langle \Uscr,\Uscr\rangle$ is the same as in said theorem.

Define $\phi_K(s) = \sqrt{n}\{\beta^{\circ}(s) - \beta^{\star}(s)\}$, so that $\beta^{\circ}(s) = \beta^{\star}(s) + \phi_K(s)/\sqrt{n}$. If $K \propto \sqrt{n}$, by the same argument used in~\eqref{eq::phi_convergence}, there is a function $\phi(s)$, such that $\int_0^t \phi_K(s)\,\dd s \to \int_0^t \phi(s)\,\dd s$ for all $t$. By setting all the $w_i(s,\theta)$ in Lemma~\ref{lemma::lr1} to one, it follows from that lemma that the log-likelihood ratio $\log(\dd P_{\circ,n}/\dd P_{\star,n}^K)$, evaluated in $t$, is 
\begin{equation}
\begin{split}
\log \frac{\dd P_{\circ,n}}{\dd P_{\star,n}^K}\bigg|_t & = \sum_{i=1}^n \{\zeta_i^K(t) - \frac{1}{2}\langle\zeta_i^K,\zeta_i^K \rangle_t\} + O_p(1/K^3),
\end{split}
\notag
\end{equation}     
as $K \to \infty$, where 
\begin{equation}
\zeta_i^K = n^{-1/2}\int_0^t a_{K,i}(s)\,\dd M_i(s,\beta^{\star}),\quad \text{with} 
\quad a_{K,i}(s) = \frac{Z_i^{\prime}\phi_K(s)}{Z_i^{\prime}\beta^{\star}(s)},
\label{eq::zeta_mg}
\end{equation} 
are $P_{\star}^K$ martingales for $i=1,\ldots,n$ and all partitions $K$. In analogy with Section~\ref{sec::shifting_back}, define $a_K(s) = \{Z^{\prime}\phi_K(s)/Z^{\prime}\beta^{\star}(s)\}$ and $a(s) = \{Z^{\prime}\phi(s)/Z^{\prime}\beta^{\circ}(s)\}$.   

\begin{lemma}\label{app::lemma_lr2} Assume that $K \propto \sqrt{n}$. When $n \to \infty$ and $K \to \infty$, 
\begin{equation}
\log \frac{\dd P_{\circ,n}}{\dd P_{\star,n}^K} \Rightarrow \log L,
\notag
\end{equation}
under $P_{\star}^K$, where $L$ is the process $L_t = \exp\{-\sigma(t)^2/2 + \int_0^t \sigma(s)\,\dd W_s \}$, with $W$ a standard Wiener process and $\sigma(t)^2 = \E\, \{Z^{\prime}\phi(t)\}^2/\{Z\beta^{\circ}(t)\} y(s;Z)$. 
\end{lemma}
\begin{proof} The quadratic variation $n^{-1}\sum_{i=1}^n \langle \zeta_i^K,\zeta_i^K\rangle_t =n^{-1}\sum_{i=1}^n \int_0^{t}a_{K,i}(s)^2Y_i(s) Z_i^{\prime}\beta^{\star}(s)\dd s$ converges in probability to $\E\,\langle \zeta_1^K,\zeta_1^K\rangle_t = \E\,  \int_0^{t}a_{K}(s)^2y(s;Z)Z^{\prime}\beta^{\star}(s) \,\dd s$. Assume w.l.o.g.~that $K = \sqrt{n}$, $\tau = 1$, that $t = v_{\ell}$, and that $\phi_K(s)$ is one-dimensional, 
\begin{equation}
\begin{split}
\langle \zeta_1^K,\zeta_1^K\rangle_t 
& = n\int_0^{v_{\ell}}Z^2\{\beta^{\circ}(s) - \beta^{\star}(s)\}^2\frac{y(s;Z)}{\beta^{\star}(s)}\,\dd s\\
& = n\sum_{j=1}^{\ell} \int_{v_{j-1}}^{v_{j}} 
Z^2\{\beta^{\circ}(s) - \beta^{\circ}(v_{j-1})\}^2\frac{y(s;Z)}{Z\beta^{\circ}(v_{j-1})}\,\dd s\\
& = n\sum_{j=1}^{\ell} Z^2\dot{\beta}^{\circ}(v_{j-1})^2 \int_{v_{j-1}}^{v_{j}} 
(s - v_{j-1})^2\frac{y(s;Z)}{Z\beta^{\circ}(v_{j-1})}\,\dd s + O_p(K/n)\\
& = n\sum_{j=1}^{\ell} Z^2\dot{\beta}^{\circ}(v_{j-1})^2 \int_{v_{j-1}}^{v_{j}} 
(s - v_{j-1})^2\frac{y(v_{j-1};Z)}{Z\beta^{\circ}(v_{j-1})}\,\dd s + O_p(K/n)\\
& = \frac{n}{3K^3}\sum_{j=1}^{\ell} Z^2\dot{\beta}^{\circ}(v_{j-1})^2 
\frac{y(v_{j-1};Z)}{Z\beta^{\circ}(v_{j-1})} + O_p(K/n)
\overset{p}\to \frac{1}{3}\int_0^{v_{\ell}} \frac{Z^2\dot{\beta}^{\circ}(s)}{Z\beta^{\circ}(v_{j-1})}y(s; Z) \,\dd s,
\end{split}
\notag
\end{equation}
as $K \to \infty$, because $n/K^3 = 1/K$. The function $\beta^{\circ}(s)$ is at least twice continuously differentiable, hence there is a $B > 0$ such that $\sup_{0 \leq s\leq 1}\abs{\dot{\beta}^{\circ}(s)} \leq B$ and $\sup_{0 \leq s\leq 1}\abs{\ddot{\beta}^{\circ}(s)}  \leq B$. Since $\abs{n\int_0^{v_{\ell}}Z^2\{\beta^{\circ}(s) - \beta^{\star}(s)\}^2y(s;Z)/\beta^{\star}(s)\,\dd s} \leq 4B^2\sum_{j=1}^{\ell}\int_{v_{j-1}}^{v_{j}}Z^2 y(s;Z)/\beta^{\star}(s)\,\dd s $, which is integrable, the dominated convergence theorem yields $\E\,\langle \zeta_1^K,\zeta_1^K\rangle_t \to \int_0^t \sigma(s)^2\,\dd s$. In this case $\phi(s) = \dot{\beta}^{\circ}(s)/3$. We have used that $y(s;Z)$ is at least one time continuously differentiable, and the assumption of $Z^{\prime}\beta^{\circ}(s) > 0$. Since $Z^{\prime}\beta^{\circ}(s)$ is continuous on $[0,\tau]$, it achieves its minimum, say $m$. By the Cauchy--Schwarz inequality, $\abs{a_{K,i}(s)} \leq (1/m)\abs{Z_i^{\prime}\sqrt{n}(\beta^{\circ}(s) - \beta^{\star}(s))} \leq (1/m)\norm{Z_i} \norm{\sqrt{n}(\beta^{\circ}(s) - \beta^{\star}(s))}\leq (q/m) \max_{l \leq q}\sup_{0 \leq s \leq \tau}\abs{\dot{\beta}_l^{\circ}(s)} \sqrt{n}/K$. So provided $\sqrt{n}/K \leq b < \infty$ for all $n$, Theorem~\ref{theorem::counting_clt} applies to $\sum_{i=1}^n\zeta_i^K$ as defined in \eqref{eq::zeta_mg}, and the results follows. 
\end{proof}
Since they have the same driving martingale, we deduce from Proposition~\ref{theorem::clt1} and Lemma~\ref{app::lemma_lr2} that we have joint convergence of 
\begin{equation}
\{\sqrt{n}(\widehat{B} - B^{\circ}),\log(\dd P_{\circ,n}/\dd P_{\star,n}^K)\},
\notag
\end{equation} 
under $(P_{\star}^K)_{K}$ as $n,K \to \infty$. By the general version of Le Cam{'}s third lemma that we used in Section~\ref{sec::shifting_back} (see e.g.~\citet[Theorem 6.6, p.~90]{vanderVaart1998} or \citet[Lemma V.1.13, p.~289]{jacod2003limit}), one finds that
\begin{equation}
\sqrt{n}(\widehat{B} - B^{\circ}) \Rightarrow  c + \Uscr,
\notag
\end{equation}
under $P_{\circ}$, provided $1/K \propto 1/\sqrt{n}$ or smaller, where $c$ is the function 
\begin{equation}
c_t = \E\,\int_0^t G(s)^{-1}z y(s;Z)Z^{\prime}\phi(s) \,\dd s\, Z^{\prime}G(s)^{-1},
\notag
\end{equation}
and $\Uscr = (\Uscr_1,\ldots,\Uscr_q)^{\prime}$ is a mean zero Gaussian martingale with 
\begin{equation}
\langle \Uscr,\Uscr\rangle_t = \E\,\int_0^{t}G(s)^{-1}Z y(s;Z)Z^{\prime}\beta^{\circ}(s) \,\dd s\, Z^{\prime}G(s)^{-1}.
\notag
\end{equation}   
If $1/K \to 0$ faster than $1/\sqrt{n}$, then $c_t$ is constant and equal to zero and we are back to the limit process of Proposition~\ref{theorem::clt1}.

\subsection{Nonparametric likelihood theory}\label{subapp::spurver}
The asymptotic theory for the proportional hazards cure model as developed by~\citet{fang2005maximum} and by~\citet{lu2008maximum} draw on the work of \citet{murphy1994consistency,murphy1995asymptotic} for the gamma frailty model. They all use a theorem due to \citet[Theorem 3.3.1, p.~310]{wellner1996weak}, also stated in~\citet[Theorem 19.26, p. 281]{vanderVaart1998}, and in~\citet[Theorem~2, p.~187]{murphy1995asymptotic}. It appears evident that this theorem could have been applied to reach the same result as we reach in Theorem~\ref{theorem::process_clt_II} (for the cure model) of this technical report. Such a theorem would allow one to circumvent going via parametric models, likelihood ratios, and contiguity, but would come with its own technical costs. In this last section we sketch how one may use nonparametric likelihood techniques, and Theorem 3.3.1 of~\citet[p.~310]{wellner1996weak} in particular, for the Aalen linear hazard estimator $\widetilde{B}(s)$ of \eqref{eq::standard.est}. In other words, we aim to reproduce the conclusions of Theorem~\ref{theorem::clt1} by other means. The remainder of this section is inspired by, and follows closely~\citet[Example 3, p.~119]{gill1989non} and~\citet[Example 2, p.~496]{breslow2015z}, who exemplify these techniques using the Nelson--Aalen estimator.    

First some notation. For independent variables $X_1,\ldots,X_n$ with distribution $P$, write $P f = \E\,f(X) = \int f\,\dd P$ and $P_n f = n^{-1}\sum_{i=1}^n f(X_i)$, for some real function $f$. A $Z$-estimator $\widehat{\theta}_n$ is the solution to a set of estimating equations $P_{n}\psi_{\theta,h} = 0$, where $h$ belongs to some set $H$. In a parametric setting $H$ is a finite index set $\{1,\ldots,p\}$, with $p$ the dimension of $\theta$; in the nonparametric setting $H$ indexes an infinite number of estimating equations.       

We have $n$ independent and identically distributed triplets of survival data $X_i = (T_i,\delta_i,Z_i)$ observed over a finite interval $[0,\tau]$; the counting process $N(t) = \delta I\{T \leq t\}$ has the additive hazard $\alpha(t) = z^{\prime}\beta(t)$ of \eqref{eq::linhazard1}; and $B_{l}(t) = \int_0^t\beta_l(s)\,\dd s$ for $l = 1,\ldots,q$. Denote the parameter space in which these $B_{1}(t),\ldots,B_{q}(s)$ live by $\Theta$. Assume that $Z_i$ and $\beta(t)$ are such that $\alpha(t)$ is positive and uniformly bounded over $[0,\tau]$, and that $\E\, [Y(s) ZZ^{\prime}]$ is positive definite for all $s\in[0,\tau]$. Define the one-dimensional parametric submodels $\beta_{\eta}(t) = \beta(t) + \eta h(t)$. Let $\text{BV}[0,\tau]$ be the set of all bounded variation functions on $[0,\tau]$. Equip this space with the norm $\norm{h}_{H} = V_0^{\tau}h + |h(0)|$, where $V_a^{b}h$ is the total variation of $h$ over $[a,b]$,\footnote{The total variation of a real function $h$ on $[a,b]$ is $V_{a}^{b}h = \sup\sum_{j=1}^{K}|h(v_{j}) - h(v_{j-1})|$, where the supremum is over all partitions $a = v_{0}<v_1 < \cdots < v_{K-1}< v_{K}= b$ of $[a,b]$. By the triangle inequality $V_0^{\tau}(h+ f) \leq V_0^{\tau}h+ V_0^{\tau}f$; that $\norm{ah}_H = |a|\norm{h}_H$ for constant $a$ is clear; and the addition of $|h(0)|$ ensures that $\norm{h}_H = 0$ implies $h = 0$, so $\norm{h}_H$ is a norm.} and let $H = \{h \in \text{BV}[0,\tau] \colon \norm{h} \leq 2\}$, and denote by $H^q$ the product space $H\times \cdots \times H$. Let $h = (h_1,\ldots,h_q)^{\prime} \in H^q$ and consider the estimating equations $P_n\psi_{B,h}  = 0$, where $\psi_{B,h}$ is the real-valued function given by
\begin{equation}
\psi_{B,h}(X)  = \int_0^{\tau} h(s)^{\prime} Z\{\dd N(s) - Y(s) Z^{\prime}\,\dd B(s)  \}. 
\label{eq::esteq2} 
\end{equation} 
We get \eqref{eq::esteq2} by following the same procedure as in \citet[p.~503]{mckeague1994partly}. Suppose that we have established that the sequence of roots $\widehat{B}_n$ of $P_n\psi_{B,h}  = 0$ is consistent for the true value $B^{\circ}$. We must first show that $\sqrt{n}(P_n \psi_{B^{\circ}} - P\psi_{B^{\circ}})$ converges to a tight Gaussian process in $\ell^{\infty}(H^q)$, with $\ell^{\infty}(H^q)$ the space of bounded real-valued functions on $H^q$ under the supremum norm $\norm{g} = \sup_{h\in H^q}|g(h)|$. In the case of \eqref{eq::esteq2} this may be done by arguing that $Zh$ and $\int h^{\prime} Y ZZ^{\prime}\,\dd B$ form Donsker classes since they are both of bounded variation (see Example 19.11 in \citet{vanderVaart1998}), and that sums of Donsker classes are Donsker. We then use weak convergence results suitable for doubly indexed martingales due to \citet{nishiyama2000weak}. By our choice of $H^q$, the conditions of \citet[Theorem 3.2, p. 697]{nishiyama2000weak} are satisfied, implying that $\sqrt{n}(P_n \psi_{B^{\circ}} - P\psi_{B^{\circ}})$ converges weakly to a Gaussian process in $\ell^{\infty}(H^q)$. 
The next step is to show that the map $B\mapsto P \psi_B$ is Fr{\'e}chet-differentiable at $B^{\circ}$, with a derivative $V \colon \Theta \to \ell^{\infty}(H^q)$ that has a continuous inverse on its range. We have that
\begin{equation}
P \psi_{B,h} = \E\,\int_0^{\tau}h^{\prime} Z \big(dN - Y Z^{\prime}\,\dd B\big) 
= - \int_0^{\tau} h^{\prime}\, \E\,[\,Y Z Z^{\prime}]\big(\dd B - \dd B^{\circ} \big).
\notag
\end{equation}
Since the map $B\mapsto P\psi_B$ is linear in $B$, that is 
\begin{equation}
V(B - B^*)h \coloneqq P\psi_{B-B^*,h} = P\psi_{B,h} - P\psi_{B^*,h}
=  - \int_0^{\tau} h^{\prime}\,\E\,[\,Y Z Z^{\prime}]\big(\dd B - \dd B^* \big),
\notag
\end{equation}
this difference equals the Fr{\'e}chet-derivative $V(B - B^{*})$. For a map $\xi$ in the range of $V$ (i.e.,~in a subset of $\ell^{\infty}(H^q)$) given by $\xi h = \int_0^{\tau} h^{\prime} \,\dd \xi$, the inverse map is thus 
\begin{equation}
V^{-1} (\xi) h = - \int_0^{\tau} h^{\prime}\,\E\,[YZ Z^{\prime}]^{-1} \,\dd \xi.
\notag
\end{equation} 
The theorem mentioned at the beginning of this section then gives that \citep[Theorem 3.3.1, p.~310]{wellner1996weak},
\begin{equation}
\sqrt{n} (\widehat{B}_n - B^{\circ})h = - V^{-1} \sqrt{n}\,P_n \psi_{B,h} + o_p(1) \Rightarrow G(V^{-1}\psi_{B^{\circ},h}), 
\notag
\end{equation}
with $G$ a Gaussian process on $\ell^{\infty}(H^q)$. We remark that we tacitly have identified the parameter space with a subset of $\ell^{\infty}(H^q)$, consisting of elements $B h = \int_0^{\tau}h^{\prime} \,\dd B$, that is, the functions in the original parameter space $\Theta$ indexes this auxiliary parameter space. Also note that
\begin{equation}
V^{-1}\psi_{B^{\circ},h} = \int_0^{\tau} h^{\prime} \E\,[YZ Z^{\prime}]^{-1} \dd\psi_{B^{\circ}}
= \int_0^{\tau} h^{\prime} \E\,[YZ Z^{\prime}]^{-1} Z \,\dd M.
\notag
\end{equation}
Let $h_t(s) = (I\{s \leq t\},\ldots,I\{s \leq t\})^{\prime}$, which is in $H^q$. One then finds that 
\begin{equation}
\sqrt{n} (\widehat{B}_n - B^{\circ})h_t =    \sum_{j=1}^q\sqrt{n} (\widehat{B}_{n,j}(t) - B_{j}^{\circ}(t)),
\notag
\end{equation} 
converges to a mean zero Gaussian process (indexed by time!) with covariance function 
\begin{equation}
\begin{split}
P\, \big(V^{-1}\psi_{B,h_t}V^{-1}  \psi_{B,h_s} \big)   &=  
P\,\big(\int_0^{\tau} h_t^{\prime}\E\,[YZ Z^{\prime}]^{-1} Z \dd M\int_0^{\tau} h_s^{\prime}  \E\,[YZ Z^{\prime}]^{-1} Z \,\dd M \big)\\
& = P \int_0^{\tau} h_t^{\prime} \E\,[YZ Z^{\prime}]^{-1} ZZ^{\prime} \,\dd\langle M,M\rangle \E\,[YZ Z^{\prime}]^{-1} h_s\\
& = \int_0^{\tau} h_t^{\prime} \,\E\,[YZ Z^{\prime}]^{-1} \E\,[ZZ^{\prime}Z^{\prime}]\,\dd B^{\circ} \,\E\,[YZ Z^{\prime}]^{-1} h_s
\end{split}
\notag
\end{equation} 
Indicator functions in the right places give the various elements of the covariance matrix, for example, $h_t(s) = (I\{s \leq t\},0,\ldots,0)^{\prime}$ gives the limiting variance of the first element of $\sqrt{n}(\widehat{B}_{n} - B^{\circ})$ and so on. From $P\, \big(V^{-1}\psi_{B,h_t}V^{-1}  \psi_{B,h_s} \big)$ above, we see that $\widehat{B}_n$ has the same limiting distribution as the estimator $\widetilde{B}$ in \eqref{eq::standard.est}. It is in fact the same estimator. 

\bibliography{refs_additiveCure}

\begin{thebibliography}{}

\bibitem[Aalen, 1980]{aalen1980model}
Aalen, O. (1980).
\newblock A model for nonparametric regression analysis of counting processes.
\newblock In {\em Lecture Notes in Statistics}, volume~2, pages 1--25. Springer
  Verlag.

\bibitem[Aalen, 1989]{aalen1989linear}
Aalen, O. (1989).
\newblock A linear regression model for the analysis of life times.
\newblock {\em Statistics in Medicine}, 8:907--925.

\bibitem[Aalen, 1993]{aalen1993further}
Aalen, O. (1993).
\newblock Further results on the non-parametric linear regression model in
  survival analysis.
\newblock {\em Statistics in Medicine}, 12:1569--1588.

\bibitem[Amico and Van~Keilegom, 2018]{amico2018cure}
Amico, M. and Van~Keilegom, I. (2018).
\newblock {C}ure models in survival analysis.
\newblock {\em Annual Review of Statistics and Its Application}, pages
  311--342.

\bibitem[Andersen et~al., 1993]{andersen1993statistical}
Andersen, P.~K., Borgan, {\O}., Gill, R.~D., and Keiding, N. (1993).
\newblock {\em Statistical {M}odels {B}ased on {C}ounting {P}rocesses}.
\newblock Springer.

\bibitem[Billingsley, 1995]{billingsley1995conv}
Billingsley, P. (1995).
\newblock {\em {P}robability and {M}easure. {T}hird {E}dition}.
\newblock John Wiley \& Sons.

\bibitem[Billingsley, 1999]{billingsley1999conv}
Billingsley, P. (1999).
\newblock {\em {C}onvergence of {P}robability {M}easures. {S}econd {E}dition}.
\newblock John Wiley \& Sons.

\bibitem[Breslow et~al., 2015]{breslow2015z}
Breslow, N.~E., Hu, J., and Wellner, J.~A. (2015).
\newblock Z-estimation and stratified samples: application to survival models.
\newblock {\em Lifetime Data Analysis}, 21:493--516.

\bibitem[Fang et~al., 2005]{fang2005maximum}
Fang, H.-B., Li, G., and Sun, J. (2005).
\newblock Maximum likelihood estimation in a semiparametric
  logistic/proportional-hazards mixture model.
\newblock {\em Scandinavian Journal of Statistics}, 32:59--75.

\bibitem[Ferguson, 1996]{ferguson1996course}
Ferguson, T.~S. (1996).
\newblock {\em A {C}ourse in {L}arge {S}ample {T}heory}.
\newblock Chapman {\&} Hall/CRC, London.

\bibitem[Gill and van~der Vaart, 1993]{gill1993non}
Gill, R.~D. and van~der Vaart, A. (1993).
\newblock Non-and semi-parametric maximum likelihood estimators and the von
  {M}ises method {I}{I}.
\newblock {\em Scandinavian Journal of Statistics}, pages 271--288.

\bibitem[Gill et~al., 1989]{gill1989non}
Gill, R.~D., Wellner, J.~A., and Pr{\ae}stgaard, J. (1989).
\newblock Non-and semi-parametric maximum likelihood estimators and the von
  {M}ises method (part 1) [with discussion and reply].
\newblock {\em Scandinavian Journal of Statistics}, pages 97--128.

\bibitem[Hermansen and Hjort, 2015]{hermansen2015bernshteuin}
Hermansen, G.~H. and Hjort, N.~L. (2015).
\newblock Bernshte{\u\i}n--von {M}ises theorems for nonparametric function
  analysis via locally constant modelling: {A} unified approach.
\newblock {\em Journal of Statistical Planning and Inference}, 166:138--157.

\bibitem[Huffer and McKeague, 1991]{huffer1991weighted}
Huffer, F.~W. and McKeague, I.~W. (1991).
\newblock {W}eighted least squares estimation for {A}alen{'}s additive risk
  model.
\newblock {\em Journal of the American Statistical Association}, 86:114--129.

\bibitem[Jacod and Shiryaev, 2003]{jacod2003limit}
Jacod, J. and Shiryaev, A. (2003).
\newblock {\em Limit {T}heorems for {S}tochastic {P}rocesses. {S}econd
  {E}dition}.
\newblock Springer, Berlin.

\bibitem[Lin and Ying, 1994]{lin1994semiparametric}
Lin, D. and Ying, Z. (1994).
\newblock Semiparametric analysis of the additive risk model.
\newblock {\em Biometrika}, 81:61--71.

\bibitem[Lu, 2008]{lu2008maximum}
Lu, W. (2008).
\newblock Maximum likelihood estimation in the proportional hazards cure model.
\newblock {\em Annals of the Institute of Statistical Mathematics},
  60:545--574.

\bibitem[McKeague and Sasieni, 1994]{mckeague1994partly}
McKeague, I.~W. and Sasieni, P.~D. (1994).
\newblock A partly parametric additive risk model.
\newblock {\em Biometrika}, 81:501--514.

\bibitem[Murphy, 1994]{murphy1994consistency}
Murphy, S.~A. (1994).
\newblock Consistency in a proportional hazards model incorporating a random
  effect.
\newblock {\em The Annals of Statistics}, 22:712--731.

\bibitem[Murphy, 1995]{murphy1995asymptotic}
Murphy, S.~A. (1995).
\newblock Asymptotic theory for the frailty model.
\newblock {\em The Annals of Statistics}, 23:182--198.

\bibitem[Mykland and Zhang, 2009]{mykland2009inference}
Mykland, P.~A. and Zhang, L. (2009).
\newblock Inference for continuous semimartingales observed at high frequency.
\newblock {\em Econometrica}, 77:1403--1445.

\bibitem[Nielsen et~al., 1992]{nielsen1992counting}
Nielsen, G.~G., Gill, R.~D., Andersen, P.~K., and S{\o}rensen, T.~I. (1992).
\newblock A counting process approach to maximum likelihood estimation in
  frailty models.
\newblock {\em Scandinavian Journal of Statistics}, 19:25--43.

\bibitem[Nishiyama, 2000]{nishiyama2000weak}
Nishiyama, Y. (2000).
\newblock {W}eak convergence of some classes of martingales with jumps.
\newblock {\em The Annals of Probability}, 28:685--712.

\bibitem[Ortega, 1990]{ortega1990numerical}
Ortega, J.~M. (1990).
\newblock {\em Numerical {A}nalysis: {A} {S}econd {C}ourse}.
\newblock SIAM.

\bibitem[Peng and Dear, 2000]{peng2000nonparametric}
Peng, Y. and Dear, K. (2000).
\newblock {A} nonparametric mixture model for cure rate estimation.
\newblock {\em Biometrics}, 56:237--243.

\bibitem[Sinha et~al., 2009]{sinha2009empirical}
Sinha, D., McHenry, M.~B., Lipsitz, S.~R., and Ghosh, M. (2009).
\newblock Empirical {B}ayes estimation for additive hazards regression models.
\newblock {\em Biometrika}, 96:545--558.

\bibitem[Stoltenberg, 2020]{stoltenberg2020kappa}
Stoltenberg, E.~A. (2020).
\newblock Kappa. {I}ntroduction to {P}h{D}-thesis.
\newblock PhD-thesis submitted at the Department of Mathematics, University of
  Oslo.

\bibitem[Sy and Taylor, 2000]{sy2000estimation}
Sy, J.~P. and Taylor, J. (2000).
\newblock {E}stimation in a {C}ox proportional hazards cure model.
\newblock {\em Biometrics}, 56:227--236.

\bibitem[van~der Vaart, 1998]{vanderVaart1998}
van~der Vaart, A. (1998).
\newblock {\em {A}symptotic {S}tatistics}.
\newblock Cambridge University Press, Cambridge.

\bibitem[van~der Vaart and Wellner, 1996]{wellner1996weak}
van~der Vaart, A. and Wellner, J. (1996).
\newblock {\em Weak {C}onvergence and {E}mpirical {P}rocesses. {W}ith
  {A}pplications to {S}tatistics}.
\newblock Springer.

\end{thebibliography}
\bibliographystyle{apalike}

\end{document}